\RequirePackage{fix-cm}
\documentclass[smallcondensed]{svjour3}     % onecolumn (ditto)
\smartqed  % flush right qed marks, e.g. at end of proof
\usepackage{mathptmx}      % use Times fonts if available on your TeX system
\usepackage{amsmath}
\usepackage{latexsym}
\usepackage{graphicx}
\usepackage{xhfill}
\newtheorem{algorithm}{Algorithm}
\begin{document}
\newcommand{\xfill}[2][1ex]{{%
  \dimen0=#2\advance\dimen0 by #1
  \leaders\hrule height \dimen0 depth -#1\hfill%
}}
\title{New inertial algorithm for a class of equilibrium problems
%\thanks{Grants or other notes
%about the article that should go on the front page should be
%placed here. General acknowledgments should be placed at the end of the article.}
}
%\subtitle{Do you have a subtitle?\\ If so, write it here}

\titlerunning{New inertial algorithm for solving strongly pseudomontone EPs}        % if too long for running head

\author{Dang Van Hieu
}

\authorrunning{Dang Van Hieu} % if too long for running head
\institute{Dang Van Hieu \at              
              Applied Analysis Research Group, Faculty of Mathematics and Statistics, \\
Ton Duc Thang University, Ho Chi Minh City, Vietnam \\
              \email{dangvanhieu@tdt.edu.vn}           %  \\
%             \emph{Present address:} of F. Author  %  if needed
            }

\date{Received: date / Accepted: date}
% The correct dates will be entered by the editor

\maketitle

\begin{abstract}
The article introduces a new algorithm for solving a class of equilibrium problems involving strongly pseudomonotone bifunctions with a 
Lipschitz-type condition. We describe how to incorporate the proximal-like regularized technique with inertial effects. The main novelty 
of the algorithm is that it can be done without previously knowing the information on the strongly pseudomonotone and Lipschitz-type constants 
of cost bifunction. A reasonable explain for this is that the algorithm uses a sequence of stepsizes which is diminishing and non-summable. 
Theorem of strong convergence is proved. In the case, when the information on the modulus of strong pseudomonotonicity and Lispchitz-type 
constant is known, the rate of linear convergence of the algorithm has been established. Several of experiments are performed to illustrate 
the numerical behavior of the algorithm and also compare it with other algorithms.
% \PACS{PACS code1 \and PACS code2 \and more}
\end{abstract}
%%%%%%%%%%%%%%%%%%%%%%%%%%%%%%%%%%

\section{Introduction}\label{intro}
The equilibrium problem (briefly, EP) \cite{BO1994,MO1992} is well known as the Ky Fan inequality early studied in \cite{F1972,NI55}. Mathematically, 
problem (EP) can be considered as a generalization of many mathematical models such as variational inequality problems, optimization problems, fixed 
point problems, complementarity problems and Nash equilibrium problems, see, e.g., \cite{BO1994,FP2002,K2007,MO1992}. So, problem (EP) 
becomes an attractive field in mathematics as well as in applied sciences. In recent years, problem (EP) has been widely studied in both theoretically and 
algorithmically. Some methods for solving problem (EP) can be found, for instance, in
\cite{A2013,A1995,BCCP13,BCP09,CH2005,FA1997,HMA2016,H2017COAP,H2016c,H2016e,H2017,H2016,K2003,K2007,M2000,M2003,M1999,MQ2009,QMH2008,SS2011,SNN2013}. 
One of the most popular methods for solving problem (EP) is the proximal point method (PPM). This method was 
first introduced by Martinet \cite{M1970} for monotone variational inequality problems and after that it was extended by Rockafellar 
\cite{R1976} to monotone operators. Moudafi \cite{M1999} extended further the PPM to EPs for monotone bifunctions. In \cite{K2003}, Konnov 
also introduced another version of the PPM with weaker assumptions.\\[.1in]
%%%%%%%%%%%%%%%%%%%%%%%%%%%%%%%%%%%%%
Another notable class of solution methods for solving problem (EP) is given by the so-called descent methods \cite{KA2006,KP2003}. 
They are based on the reformulation of the problem (EP) as a global optimization problem through the gap function or D-gap function and the 
regularization technique. The computations in these approaches often consist of evaluating the gap function at a point and searching the optimization 
direction based on the exact solution of a convex optimization problem. In recent years, the descent-like methods have been widely and intensively 
investigated under various types of weaker assumptions imposed on feasible set and cost bifunction, and also to reduce the computational complexity 
of algorithms, see, e.g., \cite{BP2015,BP2012,C2013,DPS2014}.\\[.1in]
%%%%%%%%%%%%%%%%%%%%%%%%%%%%%%%%%%%%%
Now, we are interested in a method, which is based on the auxiliary problem principle, was early introduced in \cite{FA1997} and its convergence 
was also studied. Recently, the authors in 
\cite{QMH2008} have further extended and investigated the convergence of it under different assumptions that equilibrium bifunctions are 
pseudomonotone and satisfy a certain Lipschitz-type condition \cite{M2000}. The method in \cite{FA1997,QMH2008} was also called the extragradient 
method due to the results of Korpelevich \cite{K1976} on saddle point problems. Another similar method, which is called the two-step proximal method, has 
been recently considered by the authors in \cite{LS2016}. The main advantage of this method is that it only requires to proceed a value of bifunction at the 
current approximation. Its convergence was also established under the hypotheses of pseudomonotonicity and Lipschitz-type condition of bifunctions. 
In recent years, many iterative methods based on the extragradient-like methods have been proposed for solving problem (EP) under various types of 
conditions, see, for instance \cite{HMA2016,H2016c,H2016e,LS2016,SVN2015} and the references therein.\\[.1in]
%%%%%%%%%%%%%%%%%%%%%%%%%%%%%%%%%%%%% 
It is emphasized here that the aforementioned extragradient-like methods often use stepsizes which depend on Lipschitz-type constants of equilibrium 
bifunctions. This means that the Lipschitz-type constants must be the input parameters of used method, and so the prior knowledge of these constants 
is a requirement in actual fact for constructing sequences of solution approximations. That fact can make some restrictions in applications 
because the Lipschitz-type constants are often unknown or difficult to approximate. Very recently, the works \cite{H2017AA,H2017NUMA} have 
introduced the two extragradient-like methods (with two proximal-like steps over iteration) for solving strongly pseudomonotone and Lipschitz-type equilibrium problems 
where their main advantage is that they can be done without the prior knowledge of Lipschitz-type constants and of the modulus of strong pseudomonotonicity. \\[.1in]
%%%%%%%%%%%%%%%%%%%%%%%%%%%%%%%%%%%%%
In this paper, we introduce continuously a new algorithm for solving problem (EP) involving strongly pseudomonotone bifunctions with a Lipschitz-type condition. 
As in \cite{H2017AA,H2017NUMA}, the new algorithm also can be performed in the case the information on strongly pseudomonotone and Lipschitz-type 
constant is unknown. This comes from a fact that the algorithm has used a variable sequence of stepsizes which is diminishing and non-summable. A theorem 
of strong convergence is proved. In the case, when the modulus of strong pseudomonotonicity and Lipschitz-type constants of cost bifunction are known, the rate of linear 
convergence of the algorithm is established. A notable difference in comparison with the extragradient-like methods in \cite{H2017AA,H2017NUMA} is that the proposed 
algorithm only uses a proximal-like regularized step per each iteration. In addition, the regularized step in the algorithm has been combined with inertial effects 
which has been studied recently by several authors, see, for instance, in \cite{AA2001,A2004,BCL2016,MM2008,Mo2003} and the references therein. 
As the results in \cite{AA2001,A2004,BCL2016,MM2008,Mo2003}, the extrapolation inertial term is intended to speed up the convergence properties. 
The main advantages of the new algorithm in this paper have been also confirmed by several numerical results.  \\[.1in]
%%%%%%%%%%%%%%%%%%%%%%%%%%%%%%%%%%%%%
The remainder of this paper is organized as follows: In Sect. \ref{pre} we recall some definitions and preliminary results used in the paper. 
Sect. \ref{main} introduces in details the inertial regularized algorithm and gives an estimate on the sequence generated by the algorithm. 
Sect. \ref{IRM1} analyzes the convergence of the algorithm in the case the strongly pseudomonotone and Lipschitz-type constants are 
unknown. When these constants are known, we will establishe the rate of linear convergence of the algorithm in Sect. \ref{IRM2}. 
Finally, in Sect. \ref{example} we compare the numerical behavior of the new algorithm with the regularized algorithm (without inertial effect) 
and the extragradient-like ones having the same features proposed in \cite{H2017AA,H2017NUMA}.
%%%%%%%%%%%%%%%%%%%%%%%%%%%%%%%%%%%%%%%%%%%%%%%%%%%
\section{Preliminaries}\label{pre}
%%%%%%%%%%%%%%%%%%%%%%%%%%%%%%%%%%%%%%%%
The paper concerns about solving an equilibrium problem in a real Hilbert space $H$. Let $C$ be a nonempty closed convex subset of $H$ and let $f$ be 
a bifunction from $H\times H$ to the set of real numbers $\Re$ such that $f(x,x)=0$ for all $x\in C$. Recall that the equilibrium problem (EP) for the 
bifunction $f$ on $C$ is to find $x^*\in C$ such that
$$
f(x^*,y)\ge 0,~\forall y\in C.
\eqno{\rm (EP)}
$$
Solution methods for solving problem (EP) are often relative to theory of monotonicity of an operator or a bifunction. Now, we recall some concepts 
of monotonicity of a bifunction, see \cite{BO1994,MO1992} for more details. 
%\begin{definition} 
A bifunction $f:H\times H\to \Re$ is called:\\[.1in]
(i) \textit{strongly monotone} on $C$ if there exists a constant $\gamma>0$ such that
$$ f(x,y)+f(y,x)\le -\gamma ||x-y||^2,~\forall x,y\in C; $$
(ii) \textit{monotone} on $C$ if 
$$ f(x,y)+f(y,x)\le 0,~\forall x,y\in C; $$
(iii) \textit{pseudomonotone} on $C$ if 
$$ f(x,y)\ge 0 \Longrightarrow f(y,x)\le 0,~\forall x,y\in C;$$
(iv) \textit{strongly pseudomonotone} on $C$ if there exists a constant $\gamma>0$ such that
$$ f(x,y)\ge 0 \Longrightarrow f(y,x) \le-\gamma ||x-y||^2,~\forall x,y\in C.$$
%\end{definition}
It is easy to see from the aforementioned definitions that the following implications hold,
$$ {\rm (i)}\Longrightarrow {\rm (ii)}\Longrightarrow {\rm (iii)}~{\rm and}~{\rm (i)}\Longrightarrow {\rm (iv)}\Longrightarrow {\rm (iii)}.$$
The converses in general are not true. We say that a bifunction $f:H\times H\to \Re$ satisfies \textit{Lipschitz-type condition} if there exists a 
real number $L>0$ such that 
$$
f(x,y) + f(y,z) \geq f(x,z) -L||x-y|| ||y-z||, ~ \forall x,y,z \in H. \eqno{\rm (LC)}
$$
Note that if $A:H\to H$ is a Lipschitz continuous operator, i.e., there exists $L>0$ such that $||Ax-Ay||\le L||x-y||$ for all $x,y\in H$, then the bifunction 
$f(x,y)=\left\langle Ax,y-x\right\rangle$ satisfies the Lipschitz-type condition (LC) with the constant $L$. Indeed, we have that $f(x,y) + f(y,z) - f(x,z) = 
\left\langle Ay-Ax,y-z\right\rangle\ge -||Ay-Ax||||y-z||\ge -L||y-x||||y-z||$. Thus, condition (LC) holds for $f$.
\begin{remark}\label{rem1}
The Lipschitz-type condition (LC) implies the following condition which is called the \textit{Lipschitz-type condition in the sense of Mastroeni} \cite{M2000},
$$ f(x,y) + f(y,z) \geq f(x,z) - c_1||x-y||^2 - c_2||y-z||^2, ~ \forall x,y,z \in H, \eqno{\rm (MLC)}$$
where $c_1>0,~c_2>0$ are two given constants. Indeed, if condition (LC) holds then by the following relation
$$
\left(\sqrt{\frac{L}{2\mu}}||x-y||-\sqrt{\frac{L\mu}{2}}||y-z||\right)^2\ge 0,
$$
we have 
$$ f(x,y) + f(y,z) \geq f(x,z) -L||x-y|| ||y-z||\ge f(x,z) - \frac{L}{2\mu}||x-y||^2-\frac{L\mu}{2}||y-z||^2 $$
for any $\mu>0$. This means that the Lipschitz-type condition of Mastroeni (MLC) in \cite{M2000} holds for the bifunction $f$ with $c_1=\frac{L}{2\mu}$ and $c_2=\frac{L\mu}{2}$.
\end{remark}
 Throughout this paper, for solving problem (EP), we assume that bifunction $f:H\times H\to \Re$ satisfies the following conditions:\\[.1in]
%%%%%%%%%%%%%%%%%%%%%%%
(A1) $f(x,x)=0$ for all $x\in C$;\\[.1in]
(A2) $f$ is strongly pseudomonotone on $C$ with some constant $\gamma$;\\[.1in]
(A3) $f$ satisfies the Lipschitz-type condition (LC) on $H$ with some constant $L$;\\[.1in]
(A4) $f(x,.)$ is convex and lower semicontinuous and $f(.,y)$ is hemicontinuous on $C$.\\[.1in]
%%%%%%%%%%%%%%%%%%%%%%%%%%%%%%%%%%%
Note that, under hypotheses (A2) and (A4), problem (EP) has an unique solution, denoted by $x^*$. 
The Lipschitz-type conditions are often used in establishing the convergence of extragradient-like methods for EPs, see, e.g., 
\cite{HMA2016,H2016c,H2016e,LS2016,QMH2008,SVN2015}. Recall that a function $h:C\to \Re$ is called hemicontinuous on 
$C$ if $\lim\limits_{t \to 0} h(tz + (1-t)x) = h(x)$ for all $x,~z\in C$.
The proximal mapping of a proper, convex and lower semicontinuous function $g:C\to \Re$ with a parameter $\lambda>0$ is defined by
$${\rm prox}_{\lambda g}(x)=\arg\min\left\{\lambda g(y)+\frac{1}{2}||x-y||^2:y\in C\right\},~x\in H. $$
%%%%%%%%%%%%%%%%%%%%%%%%%%%%%%%%%%%%%%%%%%%
The following is a property of the proximal mapping, see \cite{BC2011} for more details.
\begin{lemma}\label{prox}For all $x\in H,~y\in C$ and $\lambda>0$, the following inequality holds,
$$\lambda \left\{g(y)-g({\rm prox}_{\lambda g}(x))\right\}\ge \left\langle x-{\rm prox}_{\lambda g}(x),y-{\rm prox}_{\lambda g}(x)\right\rangle.$$
\end{lemma}
%%%%%%%%%%%%%%%%%%%%%%%%%%%%%%
\begin{remark}\label{rem2}
From Lemma \ref{prox}, it is easy to show that if $x={\rm prox}_{\lambda g}(x)$ then 
$$x\in {\rm arg}\min\left\{g(y):y\in C\right\}:=\left\{x\in C: g(x)=\min_{y\in C}g(y)\right\}.$$
\end{remark}
%%%%%%%%%%%%%%%%%%%%%%%%%%%%%%
The following technical lemma will be used to prove theorem of convergence in Sect. \ref{IRM1}.
\begin{lemma}\cite{AA2001}\label{lem.AA2001}  Let $\{\Phi_n\} $, $\{\Delta_n\}$ and $\{\theta_n\}$ be sequences in $[0,+\infty)$ such that 
$$
\Phi_{n+1}\le \Phi_n+\theta_n(\Phi_n-\Phi_{n-1})+\Delta_n,~ \forall n\geq 1,\quad \sum_{n=1}^{+\infty} \Delta_n<+\infty,  
$$
and there exists a real number $\theta$ with $0\le \theta_n\le \theta<1$ for all $n\ge 0$. Then the followings hold:\\[.1in]
%%%%%%%%%%%%%%%%%%%%%%%%%%%%%% 
{\rm (i)} $\sum_{n=1}^{+\infty}[\Phi_n-\Phi_{n-1}]_{+}<+\infty$, where $[t]_{+}:=\max\{t,0\}$;\\[.1in]
%%%%%%%%%%%%%%%%%%%%%%%%%%%%%%
{\rm (ii)} There exists $\Phi_*\in [0,+\infty)$ such that $\lim_{n\to +\infty}\Phi_n=\Phi_*.$
\end{lemma}
%%%%%%%%%%%%%%%%%%%%%%%%%%%%%%%%%%%%%%%%%%
Finally, in any Hilbert space, we have the following result, see, e.g., in \cite[Corollary 2.14]{BC2011}.
\begin{lemma}\label{lem3} For all $x,~y\in H$ and $\alpha\in \Re$, the following equality always holds
\begin{equation*}
\|\alpha x+(1-\alpha)y\|^2=\alpha\|x\|^2+(1-\alpha)\|y\|^2-\alpha(1-\alpha)\|x-y\|^2.
\end{equation*}
\end{lemma}
\section{Inertial regularized algorithm}\label{main}
%%%%%%%%%%%%%%%%%%%%%%%%%%%%%%%%%%%%%
This section introduces a new algorithm for solving problem (EP) involving strongly pseudomonotone and Lipschitz-type bifunctions. The algorithm can be 
considered as a combination of the proximal-like regularized technique and inertial effect. The following is the algorithm in details.\\
\noindent\rule{12.1cm}{0.4pt} 
\begin{algorithm}[Inertial Regularized Algorithm - IRA].\label{alg1}\\
\noindent\rule{12.1cm}{0.4pt}\\
\textbf{Initialization:} Choose $x_0,~x_1\in C$ and two sequences $\left\{\lambda_n\right\}\subset (0,+\infty)$ and $\left\{\theta_n\right\}\subset [0,1]$.\\
\noindent\rule{12.1cm}{0.4pt}\\
\textbf{Iterative Steps:} Assume that $x_{n-1},~x_n\in C$ are known, calculate $x_{n+1}$ as follows:

\textbf{Step 1.} Set $w_n=x_n+\theta_n(x_n-x_{n-1})$ and compute for each $n\ge 1$,
$$x_{n+1}= {\rm prox}_{\lambda_n f(w_n,.)}(w_n).$$

\textbf{Step 2.} If $x_{n+1}=w_n$ then stop and $x_{n+1}$ is the solution of problem (EP). 

Otherwise, set $n:=n+1$ and go back \textbf{Step 1}.
\end{algorithm}
\noindent\rule{12.1cm}{0.4pt}\\
%%%%%%%%%%%%%%%%%%%%%%%%%%%%%%%%
\begin{remark}\label{rem3}
The main task of Algorithm \ref{alg1} is to compute the proximal mapping in Step 1. This can be equivalently rewritten as
$$x_{n+1}= \arg\min\left\{\lambda_n f(w_n,y)+\frac{1}{2}||w_n-y||^2:y\in C\right\}.$$
Under hypotheses (A1) and (A4), from Lemma \ref{prox} and Remark \ref{rem2}, it is easy to see that if Algorithm \ref{alg1} terminates at some iterate 
$n$, i.e., $x_{n+1}=w_n$ then $x_{n+1}$ is the solution of problem (EP). Throughout the paper, we assume that Algorithm \ref{alg1} does not stop. This 
means that the sequence $\left\{x_n\right\}$ generated by Algorithm \ref{alg1} is infinite. When $\theta_n=0$, Algorithm \ref{alg1} can give us a regularized 
algorithm with a proximal-like step. As in \cite{AA2001,A2004,BCL2016,MM2008,Mo2003}, when $\theta_n\ne 0$, the extrapolation term 
$\theta_n(x_n-x_{n-1})$ is called the inertial effect and intended to speed up the convergence properties. This is also illustrated in our numerical experiments 
in the final part of this paper.
\end{remark}
Also, under hypotheses (A2) and (A4), problem (EP) has an unique solution. This unique solution will be denoted by $x^*$ in what follows. The following lemma 
will be used repeatedly in the next two sections.
\begin{lemma}\label{lem1}
Suppose that assumptions {\rm (A1) - (A4)} hold. Then, the sequence $\left\{x_n\right\}$ generated by Algorithm \ref{alg1} satisfies the following estimate,
\begin{eqnarray*}
(1+\lambda_n(2\gamma-L\sqrt{\lambda_n}))||x_{n+1}-x^*||^2&\le&(1+\theta_n)||x_n-x^*||^2-\theta_n||x_{n-1}-x^*||^2\nonumber\\
&&-M_n||x_{n+1}-x_n||^2+N_n||x_n-x_{n-1}||^2,
\end{eqnarray*}
where 
$$
M_n=(1-\theta_n)(1-L\sqrt{\lambda_n}),~N_n= \theta_n\left[1+\theta_n+(1-\theta_n)(1-L\sqrt{\lambda_n})\right].
$$
\end{lemma}
%%%%%%%%%%%%%%%%%%%%%%%%%%%%%%%%%%%%%%%
\begin{proof}
From the definitions of the proximal mapping and of $x_{n+1}$, we can write 
\begin{equation}\label{eq:1}
x_{n+1}={\rm prox}_{\lambda_n f(w_n,.)}(w_n)=\arg\min \left\{f_n(x):x\in C\right\},
\end{equation}
where $f_n(x)=\lambda_n f(w_n,x)+\frac{1}{2}||x-w_n||^2$. From relation (\ref{eq:1}) and using the optimality condition, we obtain 
$0\in \partial f_n(x_{n+1})+N_C(x_{n+1})$. Thus, there exists $g^*_n\in \partial f_n(x_{n+1})$ such that $-g_n^*\in N_C(x_{n+1})$, i.e., 
\begin{equation}\label{eq:2}
\left\langle g^*_n,x-x_{n+1}\right\rangle\ge 0,~\forall x\in C.
\end{equation}
Since $f(w_n,.)$ is convex, $f_n(x)$ is strongly convex with the modulus $1$. This implies that 
\begin{equation}\label{eq:3}
f_n(x_{n+1})+\left\langle g_n,x-x_{n+1}\right\rangle+\frac{1}{2}||x-x_{n+1}||^2\le f_n(x),~\forall x\in C,
\end{equation}
for any $g_n\in \partial f_n(x_{n+1})$. Substituting $g_n=g_n^*$ and $x=x^*$ into relation (\ref{eq:3}) and using relation (\ref{eq:2}), we get 
$$
f_n(x_{n+1})+\frac{1}{2}||x^*-x_{n+1}||^2\le f_n(x^*),
$$
which together with the definition of $f_n$ implies that 
\begin{equation}\label{eq:4}
||x_{n+1}-x^*||^2\le 2\lambda_n \left\{f(w_n,x^*)-f(w_n,x_{n+1})\right\}+||w_n-x^*||^2-||x_{n+1}-w_n||^2.
\end{equation}
Using the Lipschitz-type condition of $f$ and the Cauchy inequality, we obtain that
\begin{eqnarray}\label{eq:5}
&&f(w_n,x^*)-f(w_n,x_{n+1})\le f(x_{n+1},x^*)+L||x_{n+1}-w_n||||x_{n+1}-x^*||\nonumber\\
&\le&f(x_{n+1},x^*)+\frac{L}{2}\left(\frac{1}{\sqrt{\lambda_n}}||x_{n+1}-w_n||^2+\sqrt{\lambda_n}||x_{n+1}-x^*||^2\right).
\end{eqnarray}
Since $x^*$ is the solution of problem (EP), $f(x^*,x_{n+1})\ge 0$. Thus, from the strong pseudomotonicity of $f$, we obtain that 
$f(x_{n+1},x^*)\le -\gamma ||x_{n+1}-x^*||^2$. This together with relation (\ref{eq:5}) implies that 
$$
f(w_n,x^*)-f(w_n,x_{n+1})\le -(\gamma-\frac{L\sqrt{\lambda_n}}{2})||x_{n+1}-x^*||^2+\frac{L}{2\sqrt{\lambda_n}}||x_{n+1}-w_n||^2.
$$
Multiplying both two sides of the last inequality by $2\lambda_n$, we obtain 
\begin{eqnarray}
2\lambda_n\left(f(w_n,x^*)-f(w_n,x_{n+1})\right)&\le& -\lambda_n(2\gamma-L\sqrt{\lambda_n})||x_{n+1}-x^*||^2\nonumber\\
&&+L\sqrt{\lambda_n}||x_{n+1}-w_n||^2.\label{eq:6}
\end{eqnarray}
It follows from relations (\ref{eq:4}) and (\ref{eq:6}) that 
\begin{equation}\label{eq:7}
(1+\lambda_n(2\gamma-L\sqrt{\lambda_n}))||x_{n+1}-x^*||^2\le ||w_n-x^*||^2-(1-L\sqrt{\lambda_n})||x_{n+1}-w_n||^2.
\end{equation}
From the definition of $w_n$ and Lemma \ref{lem3} we have 
\begin{eqnarray}
||w_n-x^*||^2&=&||(1+\theta_n)(x_n-x^*)-\theta_n(x_{n-1}-x^*)||^2\nonumber\\
&=&(1+\theta_n)||x_n-x^*||^2-\theta_n||x_{n-1}-x^*||^2\nonumber\\
&&+\theta_n(1+\theta_n)||x_{n}-x_{n-1}||^2.\label{eq:8}
\end{eqnarray}
It also follows from the definition of $w_n$ that 
\begin{eqnarray}
&&||x_{n+1}-w_n||^2=||x_{n+1}-x_n-\theta_n(x_n-x_{n-1})||^2\nonumber\\ 
& =&||x_{n+1}-x_n||^2+\theta^2_n||x_n-x_{n-1}||^2-2\theta_n \left\langle x_{n+1}-x_n,x_n-x_{n-1} \right\rangle\nonumber\\
&\ge& ||x_{n+1}-x_n||^2+\theta^2_n||x_n-x_{n-1}||^2-2\theta_n ||x_{n+1}-x_n|| ||x_n-x_{n-1}||\nonumber\\
&\ge& ||x_{n+1}-x_n||^2+\theta^2_n||x_n-x_{n-1}||^2-\theta_n \left[||x_{n+1}-x_n||^2+ ||x_n-x_{n-1}||^2\right]\nonumber\\
& =&(1-\theta_n)||x_{n+1}-x_n||^2-\theta_n(1-\theta_n)||x_n-x_{n-1}||^2.\label{eq:9}
\end{eqnarray}
Combining relations (\ref{eq:7}), (\ref{eq:8}) and (\ref{eq:9}), we get 
\begin{eqnarray*}
&&(1+\lambda_n(2\gamma-L\sqrt{\lambda_n}))||x_{n+1}-x^*||^2\le(1+\theta_n)||x_n-x^*||^2-\theta_n||x_{n-1}-x^*||^2\nonumber\\
&&-(1-\theta_n)(1-L\sqrt{\lambda_n})||x_{n+1}-x_n||^2\nonumber\\
&&+\theta_n\left[1+\theta_n+(1-\theta_n)(1-L\sqrt{\lambda_n})\right]||x_n-x_{n-1}||^2,
\end{eqnarray*}
which together with the definitions of $M_n$, $N_n$ implies the desired conclusion. Lemma \ref{lem1} is proved.
\end{proof}
%%%%%%%%%%%%%%%%%%%%%%%%%%%%%%%%%
\begin{remark}\label{rem4}
In the case, when $f$ satisfies the condition (MLC) of Mastroeni in \cite{M2000} with two constants $c_1$ and $c_2$ then we have the following 
estimate
\begin{eqnarray}
(1+2\lambda_n(\gamma-c_2))||x_{n+1}-x^*||^2&\le&(1+\theta_n)||x_n-x^*||^2-\theta_n||x_{n-1}-x^*||^2\nonumber\\
&&-\bar{M}_n||x_{n+1}-x_n||^2+\bar{N}_n||x_n-x_{n-1}||^2,\label{NE}
\end{eqnarray}
where 
$
\bar{M}_n=(1-\theta_n)(1-2\lambda_n c_1),~\bar{N}_n= \theta_n\left[1+\theta_n+(1-\theta_n)(1-2\lambda_n c_1)\right].
$
Indeed, from relation (\ref{eq:4}) and the condition (MLC) of $f$ that
$$
f(w_n,x^*)-f(w_n,x_{n+1})\le f(x_{n+1},x^*)+c_1||x_{n+1}-w_n||^2+c_2||x_{n+1}-x^*||^2,
$$
we obtain 
\begin{eqnarray*}
||x_{n+1}-x^*||^2&\le&2\lambda_n\left[f(x_{n+1},x^*)+c_1||x_{n+1}-w_n||^2+c_2||x_{n+1}-x^*||^2\right]\\
&&+||w_n-x^*||^2-||x_{n+1}-w_n||^2.
\end{eqnarray*}
This together with the fact $f(x_{n+1},x^*)\le -\gamma ||x_{n+1}-x^*||^2$ implies that 
\begin{eqnarray*}
||x_{n+1}-x^*||^2&\le&2\lambda_n\left[-\gamma ||x_{n+1}-x^*||^2+c_1||x_{n+1}-w_n||^2+c_2||x_{n+1}-x^*||^2\right]\\
&&+||w_n-x^*||^2-||x_{n+1}-w_n||^2.
\end{eqnarray*}
Thus
\begin{equation}\label{eq:9a}
(1+2\lambda_n(\gamma-c_2))||x_{n+1}-x^*||^2\le ||w_n-x^*||^2-(1-2\lambda_n c_1)||x_{n+1}-w_n||^2.
\end{equation}
It follows from relations (\ref{eq:8}), (\ref{eq:9}) and (\ref{eq:9a}) that 
\begin{eqnarray*}
(1+2\lambda_n(\gamma-c_2))||x_{n+1}-x^*||^2&\le&(1+\theta_n)||w_n-x^*||^2-\theta_n||x_{n+1}-w_n||^2\\
&&-(1-\theta_n)(1-2 \lambda_n c_1)||x_{n+1}-x_n||^2\\
&&+\theta_n\left[1+\theta_n+(1-\theta_n)(1-2 \lambda_n c_1)\right]||x_n-x_{n-1}||^2,
\end{eqnarray*}
which, from the definitons of $\bar{M}_n$ and $\bar{N}_n$, is equivalent to relation (\ref{NE}).
\end{remark}
\section{Inertial regularized algorithm without prior constants}\label{IRM1}
%%%%%%%%%%%%%%%%%%%%%%%%%%%%%%%%
In this section, we consider Algorithm \ref{alg1} for solving problem (EP) for a bifunction $f$ which is strongly pseudomonotone with some modulus 
$\gamma$ (hypothesis (A2)) and satisfies the Lipschitz-type condition (LC) with some constant $L$ (hypothesis (A3)). However, as in \cite{H2017AA,H2017NUMA}, 
we will establish that Algorithm \ref{alg1} can be done without the prior knowledge of the constants $\gamma$ and $L$. This is particularly interesting when those constants 
are unknown or difficult to approximate. In order to get that result, in Algorithm \ref{alg1} we consider the sequence of stepsizes 
$\left\{\lambda_n\right\}\subset (0,+\infty)$ and the sequence of inertial parameters $\left\{\theta_n\right\}\subset [0,1]$ satisfying the following conditions: \\[.1in]
(H1): $\lim_{n\to\infty}\lambda_n=0,\qquad {\rm (H2)}:\sum_{n=0}^\infty \lambda_n=+\infty. $\\[.1in]
(H3): $\left\{\theta_n\right\}$ is non-decreasing and $\theta_n\in [0,\theta_*]$ for some $\theta_* \in [0,\frac{1}{3})$.\\[.1in]
%%%%%%%%%%%%%%%%%%%%
A simple example of sequence $\left\{\lambda_n\right\}$ satisfies conditions (H1) and (H2) as $\lambda_n=\frac{1}{(n+1)^p}$ for each $n\ge 0$, where $p\in (0,1]$. We 
have the following first main result.
\begin{theorem}\label{theo1}
Under hypotheses {\rm (A1) - (A4)} and {\rm (H1) - (H3)}, then the sequence $\left\{x_n\right\}$ generated by Algorithm \ref{alg1} converges strongly to the unique 
solution $x^*$ of problem (EP).
\end{theorem}
\begin{proof}
Since $0\le \theta^*<\frac{1}{3}$, we obtain $0\le \frac{2\theta^*}{1-\theta^*}<1$. Now let $\sigma$ be fixed in the interval $(\frac{2\theta_*}{1-\theta_*},1)$. Since $\lambda_n\to 0$, there exists $n_0\ge 0$ such that for all $n\ge n_0$,
\begin{equation}\label{eq:15a}
1-L\sqrt{\lambda_n} \ge \sigma~\mbox{and}~2\gamma-L\sqrt{\lambda_n}\ge \gamma>0.
\end{equation}
It follows from Lemma \ref{lem1} and relation (\ref{eq:15a}) that, for all $n\ge n_0,$ 
\begin{eqnarray}
||x_{n+1}-x^*||^2&\le&(1+\lambda_n(2\gamma-L\sqrt{\lambda_n}))||x_{n+1}-x^*||^2\nonumber\\
&\le& (1+\theta_n)||x_n-x^*||^2-\theta_n||x_{n-1}-x^*||^2\nonumber\\
&&-M_n||x_{n+1}-x_n||^2+N_n||x_n-x_{n-1}||^2.\label{eq:15}
\end{eqnarray}
where $M_n$ and $N_n$ are recalled that
\begin{equation}\label{eq:16}
M_n=(1-\theta_n)(1-L\sqrt{\lambda_n}),~N_n= \theta_n\left[1+\theta_n+(1-\theta_n)(1-L\sqrt{\lambda_n})\right].
\end{equation}
Let $\varphi_n=||x_n-x^*||^2-\theta_n||x_{n-1}-x^*||^2+N_n||x_n-x_{n-1}||^2$. Thus, from the non-decreasing property of $\left\{\theta_n\right\}$ and 
relation (\ref{eq:15}), we obtain that 
\begin{eqnarray}
\varphi_{n+1}-\varphi_n&=&||x_{n+1}-x^*||^2-\theta_{n+1}||x_{n}-x^*||^2+N_{n+1}||x_{n+1}-x_{n}||^2\nonumber\\ 
&&-||x_n-x^*||^2+\theta_n||x_{n-1}-x^*||^2-N_n||x_n-x_{n-1}||^2\nonumber\\ 
&\le&||x_{n+1}-x^*||^2-(1+\theta_{n})||x_{n}-x^*||^2+N_{n+1}||x_{n+1}-x_{n}||^2\nonumber\\ 
&&+\theta_n||x_{n-1}-x^*||^2-N_n||x_n-x_{n-1}||^2\nonumber\\ 
&\le&-M_n||x_{n+1}-x_n||^2+N_{n+1}||x_{n+1}-x_{n}||^2\nonumber\\ 
&=&-(M_n-N_{n+1})||x_{n+1}-x_n||^2.\label{eq:17}
\end{eqnarray}
Moreover, from the definitions of $M_n,~N_{n+1}$, relation (\ref{eq:15a}), and the facts $\sigma\in (\frac{2\theta_*}{1-\theta_*},1)$ and 
$1-L\sqrt{\lambda_{n+1}}<1$ , we have for all $n\ge n_0$ that
\begin{eqnarray*}
M_n-N_{n+1}&=&(1-\theta_n)(1-L\sqrt{\lambda_n})- \theta_{n+1}\left[1+\theta_{n+1}+(1-\theta_{n+1})(1-L\sqrt{\lambda_{n+1}})\right]\nonumber\\ 
&\ge&(1-\theta_n)\sigma-\theta_{n+1}\left[1+\theta_{n+1}+(1-\theta_{n+1})\right]\nonumber\\ 
&\ge&(1-\theta_{n+1})\sigma- 2\theta_{n+1}~\qquad \qquad (\mbox{since}~\theta_n\le \theta_{n+1})\nonumber\\
&\ge&(1-\theta_{*})\sigma- 2\theta_{*}:=K>0.
\end{eqnarray*}
This together with relation (\ref{eq:17}) implies that 
\begin{equation}\label{eq:19}
\varphi_{n+1}-\varphi_n\le -K||x_{n+1}-x_n||^2,~\forall n\ge n_0.
\end{equation}
Thus, $\left\{\varphi_n\right\}_{n=n_0}^{+\infty}$ is non-increasing. It follows from the definition of $\varphi_n$ that 
$\varphi_n\ge ||x_n-x^*||^2-\theta_n||x_{n-1}-x^*||^2$, and thus, we obtain for all $n\ge n_0$ that
\begin{eqnarray*}
||x_n-x^*||^2\le\varphi_n+\theta_n||x_{n-1}-x^*||^2\le \varphi_{n_0}+\theta_*||x_{n-1}-x^*||^2.
\end{eqnarray*}
Hence, we get by the induction that 
\begin{eqnarray*}
||x_n-x^*||^2\le \varphi_{n_0}(1+\theta_*+\ldots+\theta_*^{n-n_0})+\theta_*^{n-n_0}||x_{n_0}-x^*||^2,~\forall n\ge n_0,
\end{eqnarray*}
which implies that 
\begin{equation}\label{eq:20}
||x_n-x^*||^2\le \frac{\varphi_{n_0}}{1-\theta_*}+\theta_*^{n-n_0}||x_{n_0}-x^*||^2.
\end{equation}
It also follows from the definition of $\varphi_{n+1}$ that $\varphi_{n+1}\ge -\theta_{n+1}||x_n-x^*||^2$, and thus, from relation (\ref{eq:20}),
\begin{equation}\label{eq:21}
-\varphi_{n+1}\le \theta_{n+1}||x_n-x^*||^2\le \theta_{*}||x_n-x^*||^2\le \frac{\theta_{*}\varphi_{n_0}}{1-\theta_*}+\theta_*^{n-n_0+1}||x_{n_0}-x^*||^2.
\end{equation}
Thus, from relation (\ref{eq:19}), we obtain for all $N\ge n_0$ that 
\begin{equation}\label{eq:22}
K\sum_{n=n_0}^N ||x_{n+1}-x_n||^2\le \varphi_{n_0}-\varphi_{N+1}\le \frac{\varphi_{n_0}}{1-\theta_*}+\theta_*^{N-n_0+1}||x_{n_0}-x^*||^2.
\end{equation}
Passing to the limit in the last inequality as $N\to\infty$ and nothing that $\theta_*\in [0,\frac{1}{3})$ and $K>0$, we obtain 
\begin{equation}\label{eq:23}
\sum_{n=n_0}^\infty ||x_{n+1}-x_n||^2<+\infty,
\end{equation}
which implies that 
\begin{equation}\label{eq:24}
\lim\limits_{n\to\infty} ||x_{n+1}-x_n||^2=0.
\end{equation}
It follows from relation (\ref{eq:15}) that 
\begin{eqnarray}
||x_{n+1}-x^*||^2&\le&(1+\theta_n)||x_n-x^*||^2-\theta_n||x_{n-1}-x^*||^2\nonumber\\
&&+N_n||x_n-x_{n-1}||^2.\label{eq:25}
\end{eqnarray}
Let $\Phi_n=||x_n-x^*||^2$, $\Delta_n=N_n||x_n-x_{n-1}||^2$ and rewrite shortly inequality (\ref{eq:25}) as follows
\begin{equation}\label{eq:26}
\Phi_{n+1}\le \Phi_n+\theta_n(\Phi_n-\Phi_{n-1})+\Delta_n.
\end{equation}
Note that $\left\{N_n\right\}$ is bounded, and thus, from (\ref{eq:23}) we obtain that $\sum_{n=n_0}^\infty \Delta_n<+\infty$. This together with 
(\ref{eq:26}) and Lemma \ref{lem.AA2001} implies that $\lim\limits_{n\to\infty} \Phi_n =\Phi_*\in \Re$, i.e.,
\begin{equation}\label{eq:27}
\lim\limits_{n\to\infty} ||x_n-x^*||^2=\Phi_* \in \Re.
\end{equation}
It follows from relations (\ref{eq:7}) and (\ref{eq:15a}) that, for all $n\ge n_0$,
$$
(1+\gamma\lambda_n)||x_{n+1}-x^*||^2\le (1+\lambda_n(2\gamma-L\sqrt{\lambda_n}))||x_{n+1}-x^*||^2\le ||w_n-x^*||^2
$$
which, together with (\ref{eq:8}) and the non-decreasing property of $\left\{\theta_n\right\}$, implies that for each $n\ge n_0$,
\begin{eqnarray*}
&&\gamma \lambda_n||x_{n+1}-x^*||^2\le-||x_{n+1}-x^*||^2+||w_n-x^*||^2\\
&=&-||x_{n+1}-x^*||^2+(1+\theta_n)||x_n-x^*||^2-\theta_n||x_{n-1}-x^*||^2\\
&&+\theta_n(1+\theta_n)||x_{n}-x_{n-1}||^2\\
&=&\left[||x_n-x^*||^2-||x_{n+1}-x^*||^2\right]+\left[\theta_n||x_n-x^*||^2-\theta_n||x_{n-1}-x^*||^2\right]\\
&&+\theta_n(1+\theta_n)||x_{n}-x_{n-1}||^2\\
&\le&\left[||x_n-x^*||^2-||x_{n+1}-x^*||^2\right]+\left[\theta_n||x_n-x^*||^2-\theta_{n-1}||x_{n-1}-x^*||^2\right]\\
&&+\theta_*(1+\theta_*)||x_{n}-x_{n-1}||^2.
\end{eqnarray*}
Let $N\ge n_0$ be fixed. Using the last inequality for $n=n_0,n_0+1,\ldots,N$ and summing up these inequalities, we obtain that 
\begin{eqnarray*}
\gamma\sum_{n=n_0}^N\lambda_n||x_{n+1}-x^*||^2&\le&||x_{n_0}-x^*||^2-||x_{N+1}-x^*||^2+\theta_N||x_N-x^*||^2\\
&&-\theta_{n_0-1}||x_{n_0-1}-x^*||^2+\theta_*(1+\theta_*)\sum_{n=n_0}^N||x_{n}-x_{n-1}||^2.
\end{eqnarray*}
Passing to the limit in the last inequality as $N\to\infty$ and using relattions (\ref{eq:23}), (\ref{eq:27}) and the boundedness of $\left\{\theta_n\right\}$, we obtain that 
$$
\sum_{n=1}^\infty\lambda_n||x_{n+1}-x^*||^2<+\infty,
$$
which, together with hypothesis (H2), implies that $\lim\limits_{n\to\infty}\inf ||x_{n+1}-x^*||^2=0$. In view of relation (\ref{eq:27}), we see that the limit of 
$\left\{||x_{n+1}-x^*||^2\right\}$ exists. Thus, $\lim\limits_{n\to\infty} ||x_{n+1}-x^*||^2=0$ which completes the proof of Theorem \ref{theo1}.
\end{proof}
%%%%%%%%%%%%%%%%%%%%%%%%%%%%%%%%%%%%
Now, we consider several corollaries of Theorem \ref{theo1}. By choosing $\theta_n=0$, we obtain the following corollary.
\begin{corollary}\label{cor1}
Suppose that hypotheses {\rm (A1) - (A4)} and {\rm (H1) - (H2)} hold. Let $\left\{x_n\right\}$ be a sequence generated by the following manner: choose $x_0\in C$ 
and for each $n\ge 0$, compute 
$$ x_{n+1} ={\rm prox}_{\lambda_n f(x_n,.)}(x_n).$$
Then, the sequence $\left\{x_n\right\}$ converges strongly to the unique solution $x^*$ of problem (EP).
\end{corollary}
%%%%%%%%%%%%%%%%%%%%%%%%%%%%%%%%%%%%%%
\begin{remark}
In the case when $\theta_n=0$, we see that $N_n=0$, $M_n\ge 0$ and $2\gamma-L\sqrt{\lambda_n}\ge \gamma$ for all $n\ge n_0$. Thus, 
it follows from relation (\ref{eq:15}) that
\begin{eqnarray*}
(1+\gamma\lambda_n)||x_{n+1}-x^*||^2&<&(1+\lambda_n(2\gamma-L\sqrt{\lambda_n}))||x_{n+1}-x^*||^2\\
&\le& ||x_n-x^*||^2,~\forall n\ge n_0.
\end{eqnarray*}
This is equivalent to $||x_{n+1}-x^*||^2<\frac{1}{1+\gamma\lambda_n}||x_n-x^*||^2$. Thus, by the induction, we obtain for each $n\ge n_0$ 
that $||x_{n+1}-x^*||^2< \frac{1}{\prod_{i=n_0}^n (1+\gamma \lambda_i)}||x_{n_0}-x^*||^2$. Hence, as in \cite{H2017AA}, we come to the following estimate, 
for each $n\ge n_0$,
\begin{equation}\label{eq:28}
||x_{n+1}-x^*||^2<\frac{||x_{n_0}-x^*||^2}{1+\gamma \sum_{i=n_0}^n\lambda_i}.
\end{equation}
\end{remark}
%%%%%%%%%%%%%%%%%%%%%%%%%%%%%%%%%%%%%%
Next, we consider a special case when problem (EP) is a variational inequality problem (VIP). Let $A:H\to H$ be a nonlinear operator. The problem 
(VIP) for an operator $A$ on $C$ is to find $x^*\in C$ such that 
$$
\left\langle Ax^*,x-x^*\right\rangle \ge 0,~\forall x\in C.
\eqno{\rm (VIP)}
$$
Recall that an operator $A:C\to H$ is called: (i) Lipschitz continuous on $H$ if there exists a real number $L>0$ such that $||Ax-Ay||\le L||x-y||$ for all $x,y\in H$; (ii) 
strongly pseudomonotone on $C$ if there exists a real number $\gamma>0$ such that the following implication holds, 
$$ \left\langle Ax,y-x\right\rangle \ge 0 \Longrightarrow \left\langle Ay,x-y\right\rangle\le -\gamma ||x-y||^2,~\forall x,y\in C.$$
Let $f(x,y)=\left\langle Ax,y-x\right\rangle$ for all $x,y\in H$. Then, ${\rm prox}_{\lambda f(x,.)}(x)=P_C(x-\lambda Ax)$ for all $x,y\in H$ and $\lambda>0$, 
and if $A$ is Lipschitz continuous and strongly pseudomonotone then assumptions (A1) - (A4) hold for $f$. Thus, the following corollary follows directly from 
Theorem \ref{theo1}.
%%%%%%%%%%%%%%%%%%%%%%%%%%%%%%%%%%%%%
\begin{corollary}\label{cor2}
Suppose that $A:H\to H$ is a strongly pseudomonotone and Lipschitz continuous operator and $x^*$ is the unique solution of problem (VIP) for $A$ on $C$. 
Let $\left\{x_n\right\}$ be a sequences generated as follows: Choose $x_0,~x_1\in C$ and for each $n$ compute $w_{n}=x_n+\theta_n(x_n-x_{n-1})$ and
$$ 
x_{n+1}=P_C(w_{n}-\lambda_{n} Aw_n),
$$
where $\left\{\lambda_n\right\}\subset (0,+\infty)$, $\left\{\theta_n\right\}\subset [0,1]$ are two sequences satisfying hypotheses {\rm (H1) - (H3)}. 
Then $\left\{x_n\right\}$ converges strongly to the unique solution $x^*$ of problem (VIP).
\end{corollary}
%%%%%%%%%%%%%%%%%%%%%%%%%%%%%%%%%%%%%%%%%%%
\begin{remark}\label{rem5}
Algorithm \ref{alg1} cannot converge linearly under hypotheses (H1) and (H2). Indeed, consider our problem for $f(x,y)=x(y-x)$ for all $x,y\in C=H=\Re$ 
and Algorithm \ref{alg1} for $\theta_n=0$ and $\lambda_n\ne 1$ for all $n\ge 0$. The unique solution of the problem is $x^*=0$. From Algorithm \ref{alg1} 
we obtain $x_{n+1}=(1-\lambda_n)x_n$. Since $\lim\limits_{n\to\infty}\lambda_n =0$ and $x_n\ne 0$ for all $n\ge 0$, we have
$$ \lim\limits_{n\to\infty}\frac{||x_{n+1}-x^*||}{||x_{n}-x^*||}=\lim\limits_{n\to\infty}|1-\lambda_n|=1. $$
Thus, we cannot find any real number $\alpha\in (0,1)$ such that $||x_{n+1}-x^*||\le \alpha ||x_{n}-x^*||$ for each $n\ge 0$. This says that Algorithm \ref{alg1} 
cannot be linearly convergent. In the next section, we will establish the rate of linear convergence of Algorithm \ref{alg1} when the strongly pseudomonotone 
and Lipschitz-type constants are known.
\end{remark}
%%%%%%%%%%%%%%%%%%%%%%%%%%%%%%%%%%%%%%%%%
\section{Inertial regularized algorithm with prior constants}\label{IRM2}
%%%%%%%%%%%%%%%%%%%%%%%%%%%%%%%%%%%
This section also studies the convergence of Algorithm \ref{alg1} under hypotheses (A2) and (A3). However, unlike the previous section, we consider the case 
when the modulus of strong pseudomonotonicity $\gamma$ and the Lipschitz-type constant $L$ are known. In that case, we establish the rate of linear convergence 
of Algorithm \ref{alg1}. For the sake of simplicity, in Algorithm \ref{alg1}, we consider that $\lambda_n=\lambda$, $\theta_n=\theta$ for all $n\ge 0$. In order to 
obtain the rate of convergence of the algorithm, we consider the following assumptions:\\[.1in]
(H4) $0<\lambda<\min\left\{\frac{4\gamma^2}{L^2},\frac{1}{L^2}\right\}$.\\[.1in]
(H5) $0\le \theta < \min\left\{\lambda(2\gamma-L\sqrt{\lambda}),\frac{1-L\sqrt{\lambda}}{3-L\sqrt{\lambda}+2\lambda(2\gamma-L\sqrt{\lambda})}\right\}$.\\[.1in]
%%%%%%%%%%%%%%%%%%%%%%%%%%%%%%
We have the following second main result.
\begin{theorem}\label{theo2}
Under hypotheses {\rm (A1) - (A4)} and {\rm (H4) - (H5)}, then the sequence $\left\{x_n\right\}$ generated by Algorithm \ref{alg1} converges linearly to the 
unique solution $x^*$ of problem (EP). Moreover, there exists $M>0$ such that for all $n\ge 1$, 
$$ ||x_{n+1}-x^*||\le M\alpha^n, $$
where 
$$\alpha=\sqrt{\frac{1+\theta}{1+\lambda(2\gamma-L\sqrt{\lambda})}}\in (0,1).$$
\end{theorem}
\begin{proof}
It follows from Lemma \ref{lem1} with $\lambda_n=\lambda$, $\theta_n=\theta$ for all $n\ge 0$ that 
\begin{eqnarray*}
&&(1+\lambda(2\gamma-L\sqrt{\lambda}))||x_{n+1}-x^*||^2+(1-\theta)(1-L\sqrt{\lambda})||x_{n+1}-x_n||^2\\
&&\le(1+\theta)||x_n-x^*||^2+\theta\left[1+\theta+(1-\theta)(1-L\sqrt{\lambda})\right]||x_n-x_{n-1}||^2.
\end{eqnarray*}
Dividing both two sides of the last inequality by $1+\lambda(2\gamma-L\sqrt{\lambda})>0$, we obtain that
\begin{equation}\label{eq:11}
||x_{n+1}-x^*||^2+B ||x_{n+1}-x_n||^2\le A ||x_n-x^*||^2+C ||x_n-x_{n-1}||^2,
\end{equation}
where
\begin{eqnarray}
&&A=\frac{1+\theta}{1+\lambda(2\gamma-L\sqrt{\lambda})},~ B=\frac{(1-\theta)(1-L\sqrt{\lambda})}{1+\lambda(2\gamma-L\sqrt{\lambda})},\label{eq:12}\\
&& C= \frac{\theta\left[1+\theta+(1-\theta)(1-L\sqrt{\lambda})\right]}{1+\lambda(2\gamma-L\sqrt{\lambda})}.\label{eq:13}
\end{eqnarray}
Under hypotheses (H4) - (H5), we see that $A>0$, $B>0$ and $C\ge 0$. Relation (\ref{eq:11}) can be rewritten as follows:
\begin{eqnarray}\label{eq:14}
||x_{n+1}-x^*||^2+B ||x_{n+1}-x_n||^2&\le& A \left(||x_n-x^*||^2+B ||x_n-x_{n-1}||^2\right)\nonumber\\
&&-(AB-C)||x_n-x_{n-1}||^2.
\end{eqnarray}
Now, under hypothesis (H4) and (H5) we will imply that $AB\ge C$. Indeed, it follows from (H5) that $\theta (3-L\sqrt{\lambda}
+2\lambda(2\gamma-L\sqrt{\lambda}))\le1-L\sqrt{\lambda}$. Thus, since  $3-L\sqrt{\lambda}=(1-L\sqrt{\lambda})+2$, we obtain 
$2\theta(1+\lambda(2\gamma-L\sqrt{\lambda}))\le (1-\theta)(1-L\sqrt{\lambda})$. This together with the fact $1-\theta\le 1-\theta^2$ implies that 
$2\theta(1+\lambda(2\gamma-L\sqrt{\lambda}))\le (1-\theta^2)(1-L\sqrt{\lambda})=(1+\theta)(1-\theta)(1-L\sqrt{\lambda})$. Multiplying both 
two sides of this inequality by $\frac{1}{(1+\lambda(2\gamma-L\sqrt{\lambda}))^2}$, we come to the following estimate
$$ \frac{(1+\theta)(1-\theta)(1-L\sqrt{\lambda})}{\left(1+\lambda(2\gamma-L\sqrt{\lambda})\right)^2}\ge \frac{2\theta}{1+\lambda(2\gamma-L\sqrt{\lambda})}.$$
Thus, since $2\theta=\theta\left[1+\theta+(1-\theta)\right]\ge \theta\left[1+\theta+(1-\theta)(1-L\sqrt{\lambda})\right]$, one has 
$$ \frac{(1+\theta)(1-\theta)(1-L\sqrt{\lambda})}{\left(1+\lambda(2\gamma-L\sqrt{\lambda})\right)^2} 
\ge\frac{\theta\left[1+\theta+(1-\theta)(1-L\sqrt{\lambda})\right]}{1+\lambda(2\gamma-L\sqrt{\lambda})}.$$
This together with the definitions of $A,~B,~C$ is equivalent to the inequality $AB\ge C$ or $A B-C\ge 0$. Thus, from relation (\ref{eq:14}), we obtain 
\begin{eqnarray*}
||x_{n+1}-x^*||^2+B ||x_{n+1}-x_n||^2&\le& A \left(||x_n-x^*||^2+B ||x_n-x_{n-1}||^2\right).
\end{eqnarray*}
Therefore, we obtain by the induction that 
\begin{eqnarray*}
||x_{n+1}-x^*||^2+B ||x_{n+1}-x_n||^2&\le& A^n \left(||x_1-x^*||^2+B ||x_1-x_{0}||^2\right),~\forall n\ge 0.
\end{eqnarray*}
Thus $||x_{n+1}-x^*||^2\le A^n \left(||x_1-x^*||^2+B ||x_1-x_{0}||^2\right)$, i.e.,
\begin{eqnarray*}
||x_{n+1}-x^*||\le M \alpha^n,
\end{eqnarray*}
where $M=\sqrt{||x_1-x^*||^2+B ||x_1-x_{0}||^2}$ and 
$$\alpha=\sqrt{A}=\sqrt{\frac{1+\theta}{1+\lambda(2\gamma-L\sqrt{\lambda})}}.$$
Note that from hypothesis (H5) we obtain that $\alpha\in (0,1)$. Theorem \ref{theo2} is proved.
\end{proof}
%%%%%%%%%%%%%%%%%%%%%%%%%%%%%%%%%%%%%%%
\begin{remark}\label{rem6}
In view of Remark \ref{rem4} and the proofs of Theorem \ref{theo1} and \ref{theo2}, we can establish the same convergence results for 
equilibrium problem (EP) with the Lipschitz-type condition (MLC) in \cite{M2000} under some suitable conditions imposed on stepsize as well as inertial 
parameter. It is worth mentioning that from the left-hand side of inequality (\ref{NE}), we always need the condition $\gamma>c_2$. This 
condition was also used in the regularized method, see, e.g., \cite[Corollary 2.1]{MQ2009}. We leave the proof in details to the readers.
\end{remark}
%%%%%%%%%%%%%%%%%%%%%%%%%%%%%%%%%%%%%%%
\section{Numerical illustrations}\label{example}
%%%%%%%%%%%%%%%%%%%%%%%%%%%%%%%%%%%
This section presents several experiments to illustrate the numerical behavior of the proposed algorithm - IRA (Algorithm \ref{alg1}) with different 
parameters, and also to compare with three other algorithms having the same features, namely the regularized algorithm - RA (see, Corollary \ref{cor1}), 
the extragradient method (EGM) presented in \cite[Algorithm 1]{H2017AA} and the modified extragradient method (M-EGM) proposed in 
\cite[Algorithm 3.1]{H2017NUMA}. As in \cite{H2017AA,H2017NUMA}, the main advantage of Algorithm \ref{alg1} is that it can be done without 
the prior knowledge of strongly pseudomonotone and Lipschitz-type constants of cost bifunction. This, as mentioned above, comes from the use of 
sequences of stepsizes being diminishing and non-summable. We use the following five sequences of stepsizes,
$$ \lambda_n=\frac{1}{(n+1)^p},~p=1,~0.7, ~0.5,~0.3,~0.1, $$
and five inertial parameters as $\theta_n\in \left\{0.1,~0.15,~0.2,~0.25,~0.3\right\}$. In the case, when the solution of the problem is unknown we use the function 
$$D(x)=||x-\mbox{\rm prox}_{\lambda f(x,.)}(x)||^2$$
for some $\lambda>0$ to describe and compare the computational performance of all the algorithms. Note that if $D(x)=0$ then $x$ is the solution of 
the problem. Otherwise, if the solution $x^*$ of the problem is known, we use the function $E(x)=||x-x^*||^2$. All the programs are written in Matlab 7.0 and computed 
on a PC Desktop Intel(R) Core(TM) i5-3210M CPU @ 2.50GHz, RAM 2.00 GB. 
%%%%%%%%%%%%%%%%%%%%%%%%%%%%%%%%%%%%%%%%%%%%%%%%%%
\subsection{Numerical behavior of Algorithm \ref{alg1}}
%%%%%%%%%%%%%%%%%%%%%%%%%%%%%%%%%%%%%%%%%%%%%%%%%%
This subsection studies the numerical behavior of Algorithm \ref{alg1} on two test problems for different control parameters. The followings 
are the examples in details. \\[.1in]
%%%%%%%%%%%%%%%%%%%%%%%%%%%%%%%%%%%%%%%%%%%%%%%%%%
\textbf{Example 1.} In this example, we consider a test problem generalized from the Nash-Cournot oligopolistic equilibrium model in 
\cite{CKK2004,FP2002} with the price and fee-fax functions being affine. The test problem is descibed as follows (also, see 
\cite{H2017AA,H2017NUMA,HMA2016}): 
Assume that there are $m$ companies that produce a commodity. Let $x$ denote the vector whose entry $x_j$ stands for the quantity  of
 the commodity produced  by company $j$. We suppose that the price $p_j(s)$ is a decreasing affine function of $s$ with $s= \sum_{j=1}^m x_j$, 
i.e., $p_j(s)= \alpha_j - \beta_j s$, where $\alpha_j > 0$, $\beta_j > 0$. Then the profit made by company $j$ is given by $f_j(x)= p_j(s) x_j -
c_j( x_j)$, where $c_j(x_j)$ is the tax and fee for generating $x_j$. Suppose that $C_j=[x_j^{\min},x_j^{\max}]$ is the strategy set of company $j$, then the
strategy set of the model is $C:= C_1\times C_2 \times ...\times C_m$. Actually, each company seeks to  maximize its profit by choosing the
corresponding production level under the presumption that the production of the other companies is a parametric input.
 A commonly used approach to this model is based upon the famous Nash equilibrium concept. We recall  that a point $x^* \in C=C_1\times C_2
\times\cdots\times C_m$ is  an equilibrium point of the model   if
    $$f_j(x^*) \geq f_j(x^*[x_j]) \ \forall x_j \in C_j, \ \forall  j=1,2,\ldots,m,$$
    where the vector $x^*[x_j]$ stands for the vector obtained from
    $x^*$ by replacing $x^*_j$ with $x_j$.
By taking $f(x, y):= \psi(x,y)-\psi(x,x)$ with $\psi(x,y):=  -\sum_{j = 1}^m f_j(x[y_j])$, 
the problem of finding a Nash equilibrium point of the model can be formulated as:
$$\mbox{Find}~x^* \in C~\mbox{such that}~ f(x^*,x) \geq 0 ~ \forall x \in C. $$
Now, assume that the tax-fee function $c_j(x_j)$ is increasing and affine for every $j$. This assumption means that both of the tax and fee  
for producing a unit are increasing as the quantity of the production gets larger. In that case, the bifunction $f$ can be formulated in 
the form 
$$f(x,y)=\left\langle Px+Qy+q,y-x\right\rangle,$$
where $q\in \Re^m$ and $P,~Q$ are two matrices of order $m$ such that $Q$ is symmetric positive semidefinite and $Q-P$ is symmetric 
negative semidefinite. However, unlike in \cite{H2016c,H2016e,H2017,QMH2008}, we consider here that $Q-P$ is symmetric 
negative \textit{definite}. From the property of $Q-P$, 
if $f(x,y)\ge 0$, we have 
\begin{eqnarray*}
f(y,x)&\le&f(y,x)+f(x,y)= (x-y)^T(Q-P)(x-y)\le-\gamma ||x-y||^2,
\end{eqnarray*}
where some $\gamma>0$. This shows that $f$ is strongly pseudomonotone, i.e., (A2) holds for $f$. Also, from the symmetric property of $Q$ 
and a straightforward computation, 
we obtain $f(x,y)+f(y,z)-f(x,z)=(y-x)^T(P-Q)(z-y)\ge -||P-Q|| ||y-x|| ||z-y||$. Thus, $f$ satisfies the condition (LC). The hypotheses (A1) and (A4) 
are automatically satisfied. A more general form of the bifunction $f$ above has been presented in \cite{QMH2008} and hypotheses (A2) and 
(A3) were also implied in details in \cite[Lemmas 6.1 and 6.2]{QMH2008}. Then, Algorithm \ref{alg1} can be applied in this case. For experiments, our problem is done in 
$\Re^{m}$ with $m=100$; the feasible set is a polyhedral set, given by 
%%%%%%%%%%%%%%%%%%%%%%%%%%%%%%%%%
$$C=\left\{x\in \Re^m_+:Ax\le b\right\},$$ 
where $A$ is a random maxtrix of size $l\times m$ with $l=10$, and the vector $b\in \Re^l_+$ is chosen such that the following point $x_0\in C$. 
The starting points are $x_0=x_1=(1,1,\ldots,1)^T\in \Re^m$. The datas are as follows: all the entries of $q$ is generated randomly 
and uniformly in $(-2,2)$ and the two matrices $P,~Q$ are also 
generated randomly\footnote{We randomly choose $\lambda_{1k}\in (-2,0),~\lambda_{2k}\in (0,2),~ k=1,\ldots,m$. We set $\widehat{Q}_1$, $\widehat{Q}_2$ 
as two diagonal matrixes with eigenvalues $\left\{\lambda_{1k}\right\}_{k=1}^m$ and $\left\{\lambda_{2k}\right\}_{k=1}^m$, respectively. Then, we 
construct a positive semidefinite matrix $Q$ and a negative definite matrix $T$ by using random orthogonal matrixes with $\widehat{Q}_2$ 
and $\widehat{Q}_1$, respectively. Finally, we set $P=Q-T$} such that their conditions hold. All the optimization subproblems are effectively solved 
by the function \textit{quadprog} in Matlab. Figs. \ref{fig1} - \ref{fig4} show the numerical behavior of algorithm IRA in this example for several different 
inertial parameters and sequences of stepsizes.
%%%%%%%%%%%%%%%%%%%%%%%%%%%%
\begin{figure}[!ht]
\begin{minipage}[b]{0.45\textwidth}
\centering
\includegraphics[height=5cm,width=6cm]{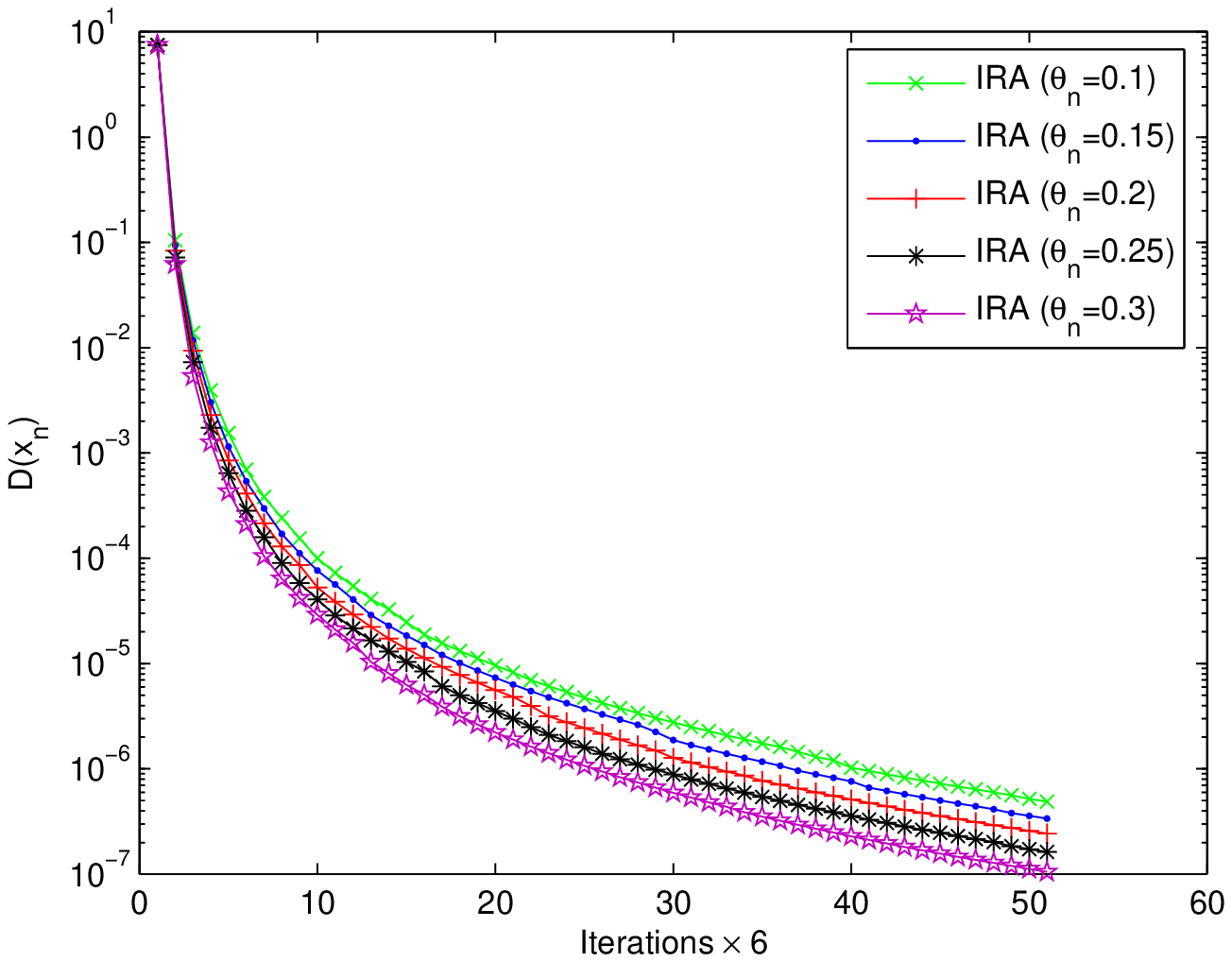}
\caption{Example 1 with $\lambda_n=\frac{1}{n+1}$.}\label{fig1}
\end{minipage}
\hfill
\begin{minipage}[b]{0.45\textwidth}
\centering
\includegraphics[height=5cm,width=6cm]{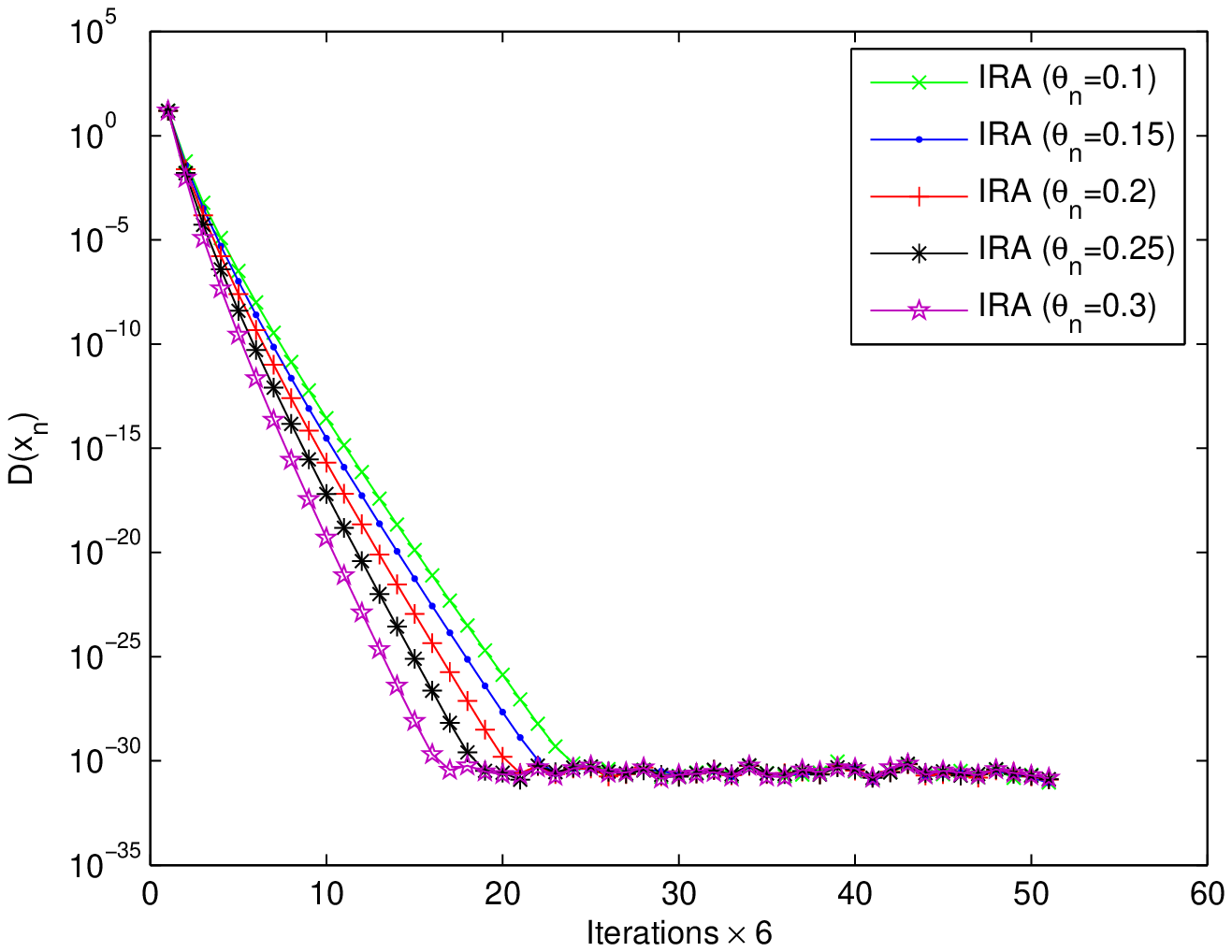}
\caption{Example 1 with $\lambda_n=\frac{1}{\sqrt[10]{n+1}}$.}\label{fig2}
\end{minipage}
\end{figure}

\begin{figure}[!ht]
\begin{minipage}[b]{0.45\textwidth}
\centering
\includegraphics[height=5cm,width=6cm]{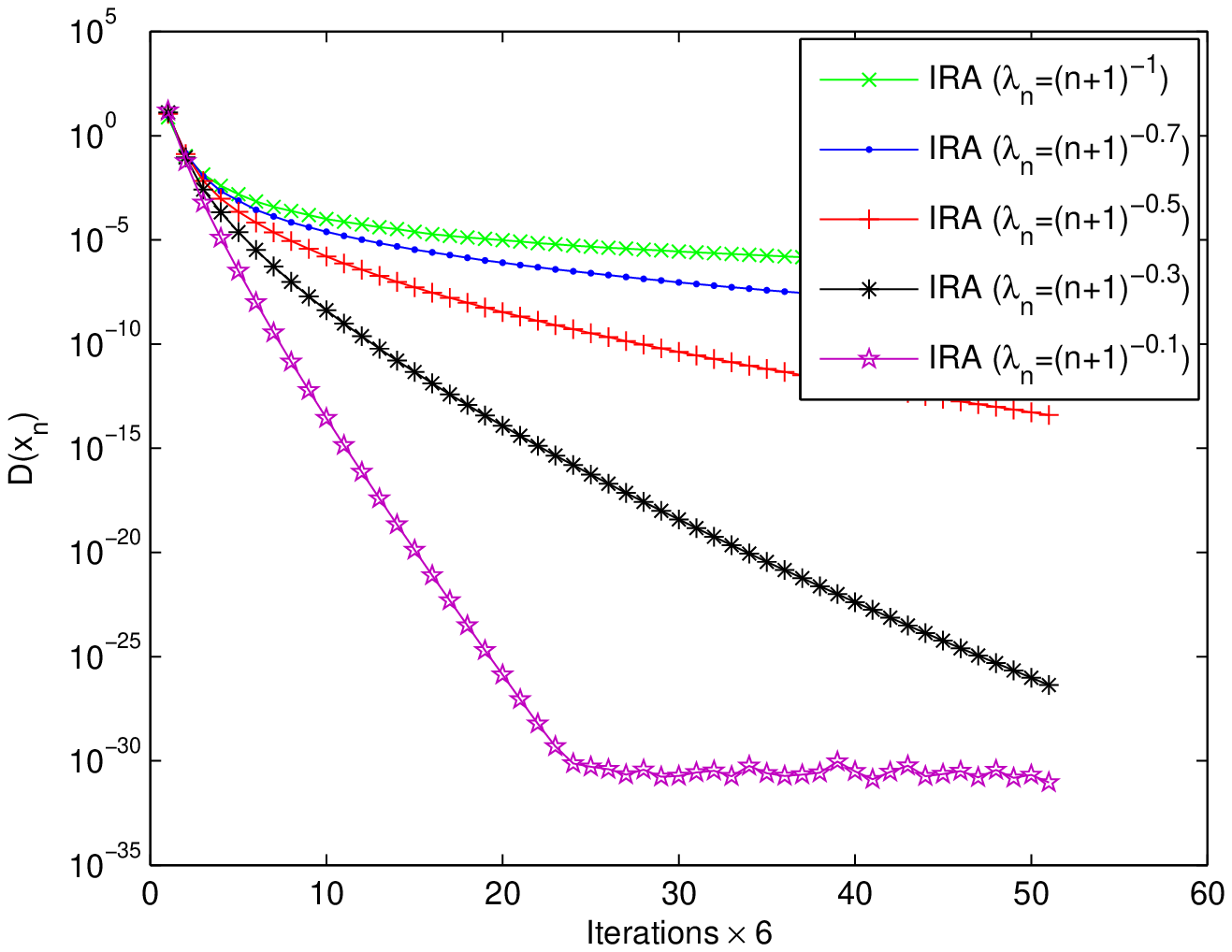}
\caption{Example 1 with $\theta_n=0.1$.}\label{fig3}
\end{minipage}
\hfill
\begin{minipage}[b]{0.45\textwidth}
\centering
\includegraphics[height=5cm,width=6cm]{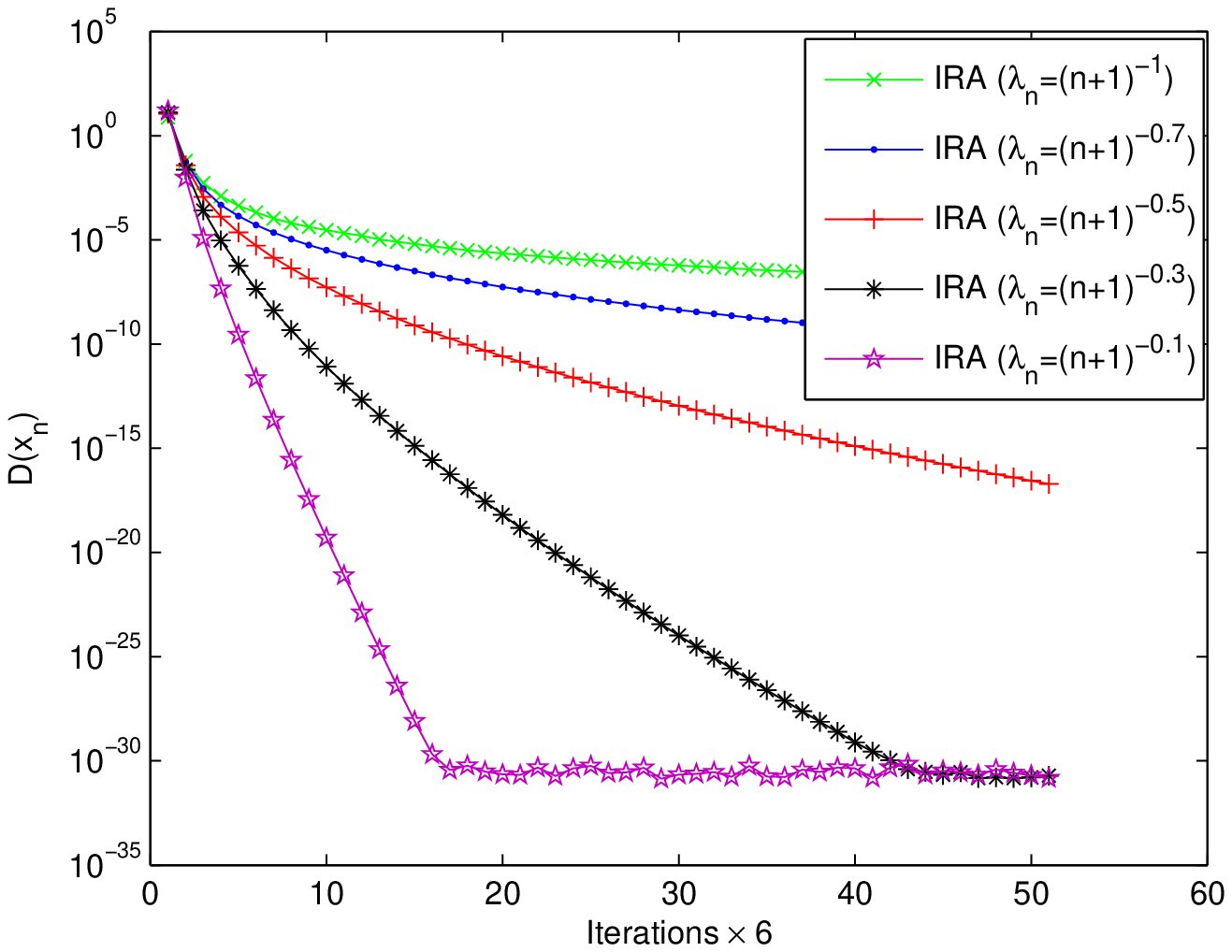}
\caption{Example 1 with $\theta_n=0.3$.}\label{fig4}
\end{minipage}
\end{figure}
%%%%%%%%%%%%%%%%%%%%%%%%%%%%%%%%%%%

\noindent\textbf{Example 2.} Now, consider the equilibrium problem in the Hilbert space $H=L^2[0,1]$ with the inner product 
$\left\langle x, y\right\rangle = \int_0^1 x(s)y(s)ds$ and the induced norm $||x||^2=\int_0^1 x^2(s)ds$. The feasible set $C$ is the 
unit ball $\mbox{B}[0,1]$ and the bifunction $f$ is of the form $f(x,y)=\left\langle Ax,y-x\right\rangle$ with the operator $A:H\to 
H$ defined by
\begin{equation}\label{eq:45}
A(x)(t)=\int_{0}^1 \left[x(t)-F(t,s)f(x(s))\right]ds+g(t),~x\in H,~t\in [0,1],
\end{equation}
where
$$
F(t,s)=\frac{2tse^{t+s}}{e\sqrt{e^2-1}},~ f(x)=\cos x,~ g(t)=\frac{2te^t}{e\sqrt{e^2-1}}. 
$$
Note that $g(t)$ is chosen such that $x^*(t)=0$ is the solution of the problem. Since the mapping $S(x)(t)=\int_{0}^1 F(t,s)f(x(s))ds$ 
is Fr$\rm \acute{e}$chet differentiable and $||S'(x)h||\le ||x||||h||$ for all $x,~h\in H$. Thus, a straightforward computation implies that 
$f$ is monotone (so, pseudomonotone) and satisfies the Lipschitz-type condition. We do not know whether $f$ is strongly pseudomonotone 
or not?!, but we still wish to make numerical experiments for this example, and a fact that if yes, we also do not need to know the Lipschitz-type 
and strongly pseudomonotone constants of $f$. All the optimization problems in the algorithms are reduced to the projections on $C$ which 
are explicitly computed. The integral in (\ref{eq:45}) and others are computed by the trapezoidal formula with the stepsize $\tau=0.001$. The 
starting points are $x_0(t)=x_1(t)=t+0.5\cos t$. The numerical results are described in Figures \ref{fig5} - \ref{fig8}.

\begin{figure}[!ht]
\begin{minipage}[b]{0.45\textwidth}
\centering
\includegraphics[height=5cm,width=6cm]{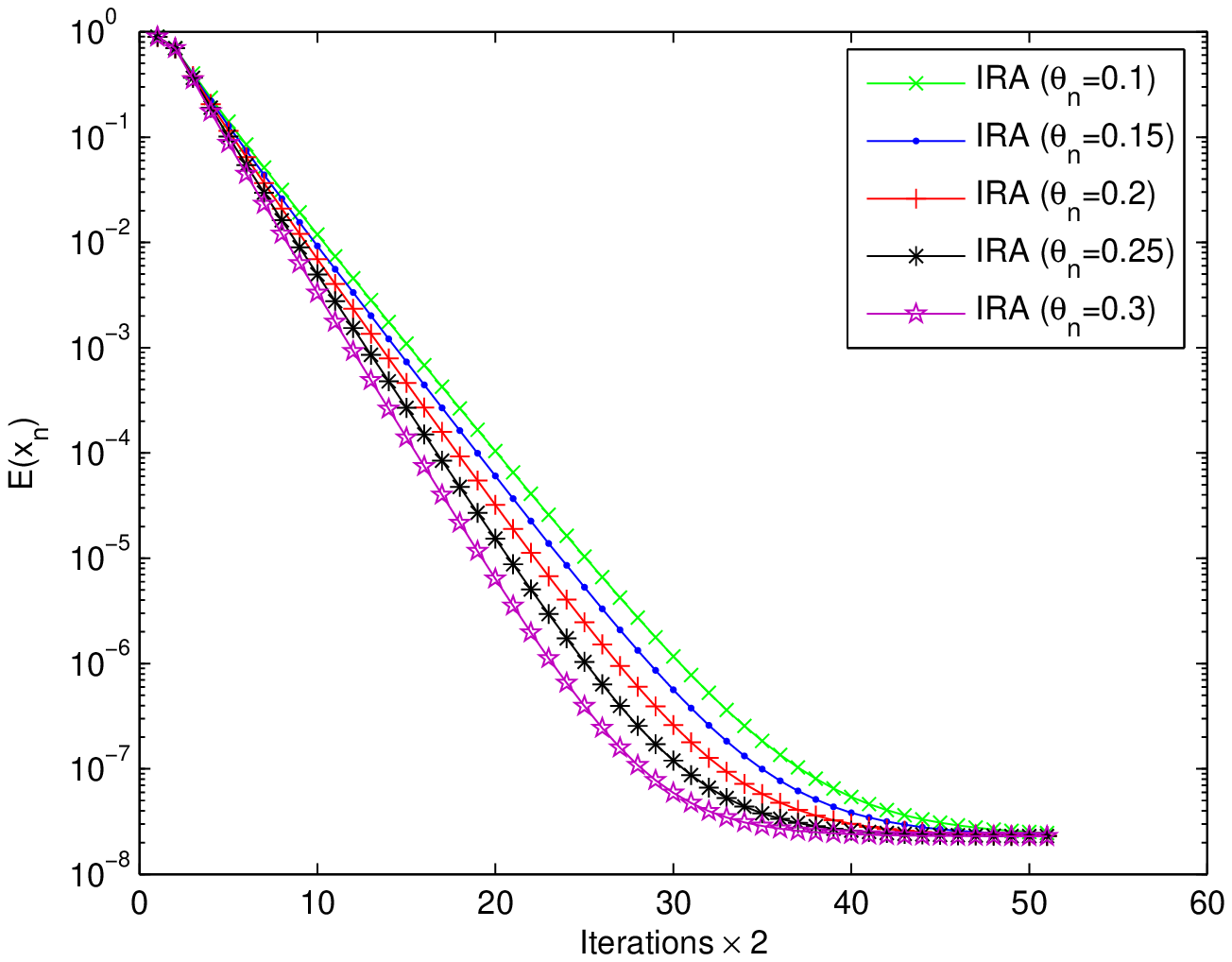}
\caption{Example 2 with $\lambda_n=\frac{1}{n+1}$.}\label{fig5}
\end{minipage}
\hfill
\begin{minipage}[b]{0.45\textwidth}
\centering
\includegraphics[height=5cm,width=6cm]{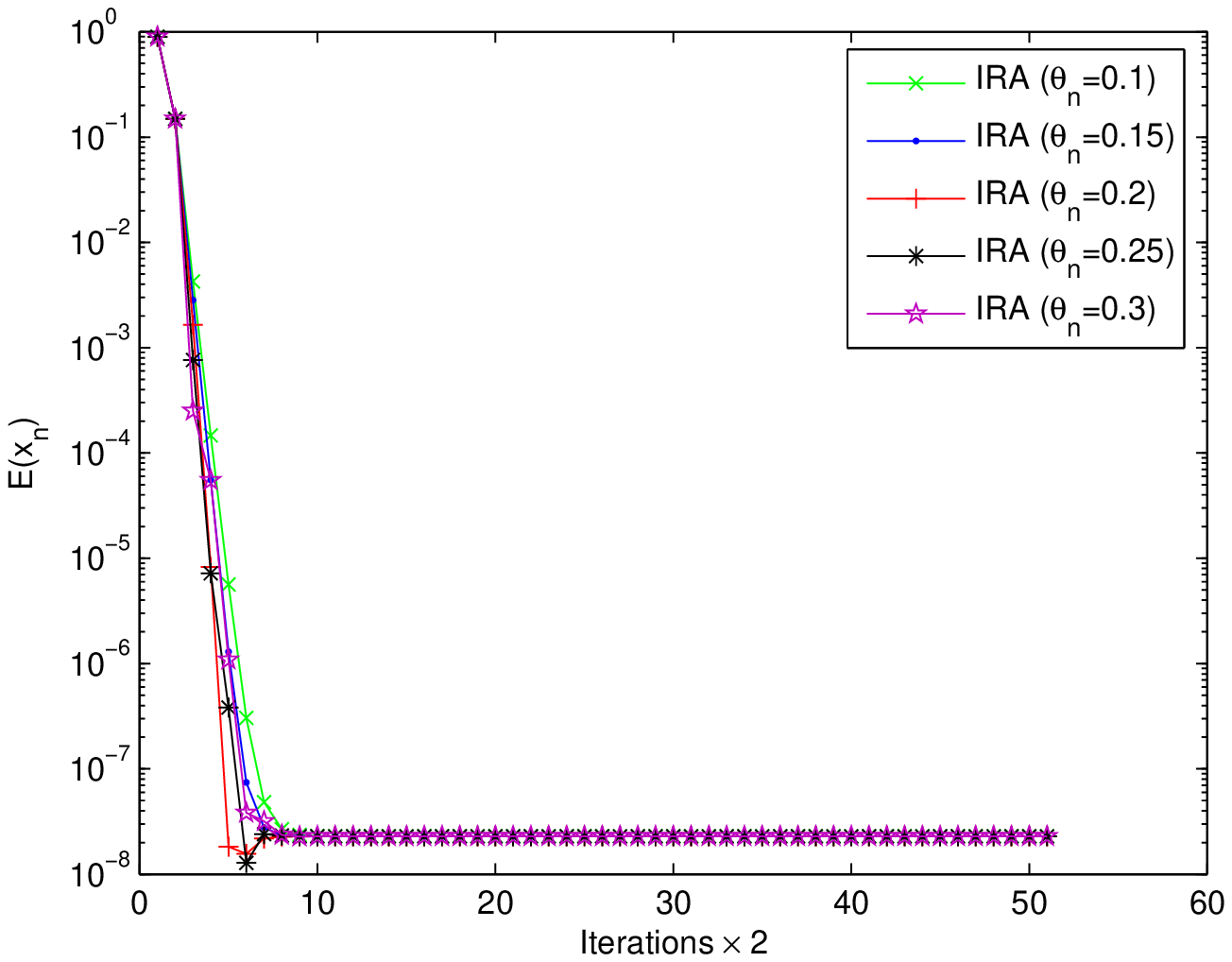}
\caption{Example 2 with $\lambda_n=\frac{1}{\sqrt[10]{n+1}}$.}\label{fig6}
\end{minipage}
\end{figure}
%%%%%%%%%%%%%%%%%%%%%%%%%%%%%%%%%%%
\begin{figure}[!ht]
\begin{minipage}[b]{0.45\textwidth}
\centering
\includegraphics[height=5cm,width=6cm]{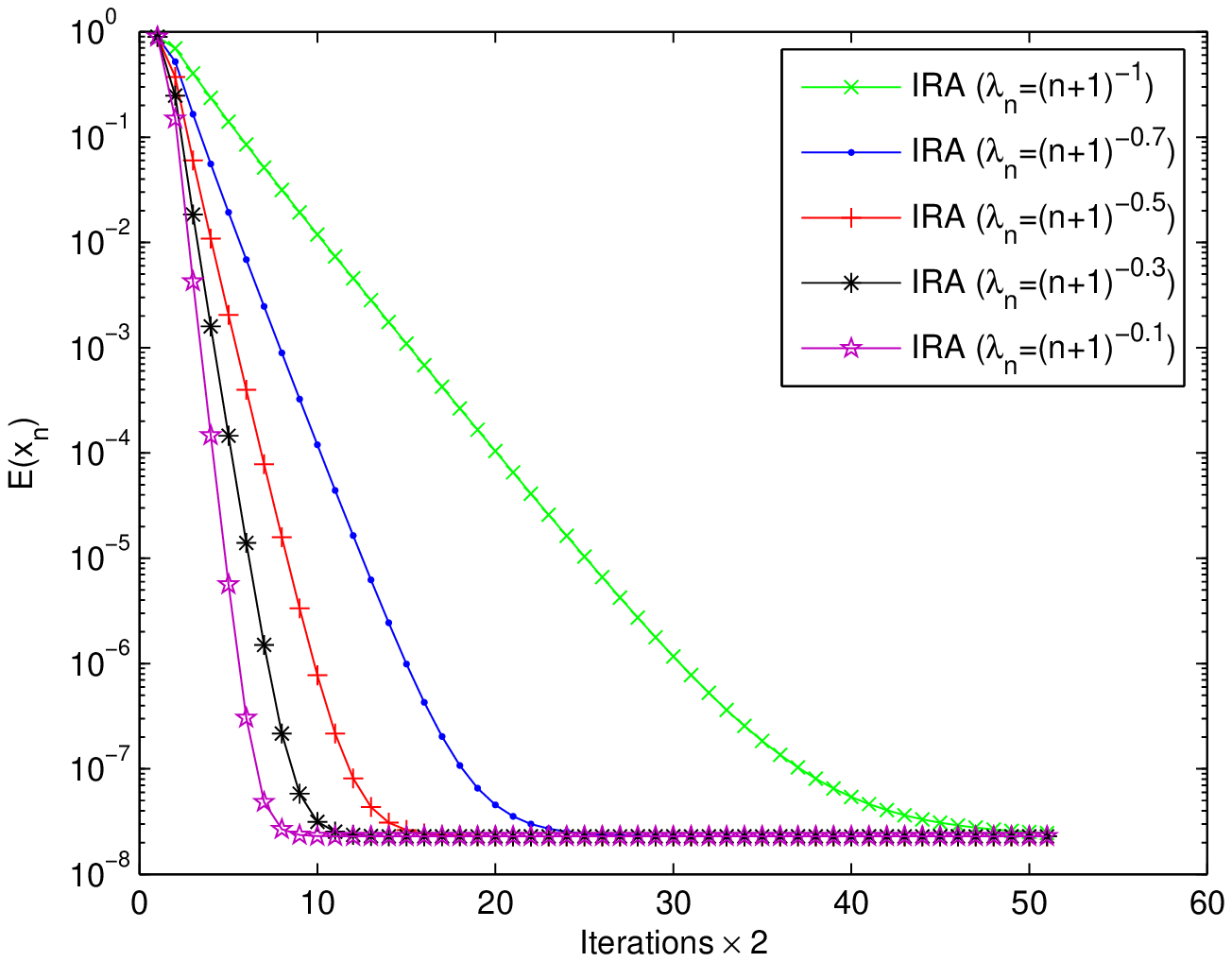}
\caption{Example 2 with $\theta_n=0.1$.}\label{fig7}
\end{minipage}
\hfill
\begin{minipage}[b]{0.45\textwidth}
\centering
\includegraphics[height=5cm,width=6cm]{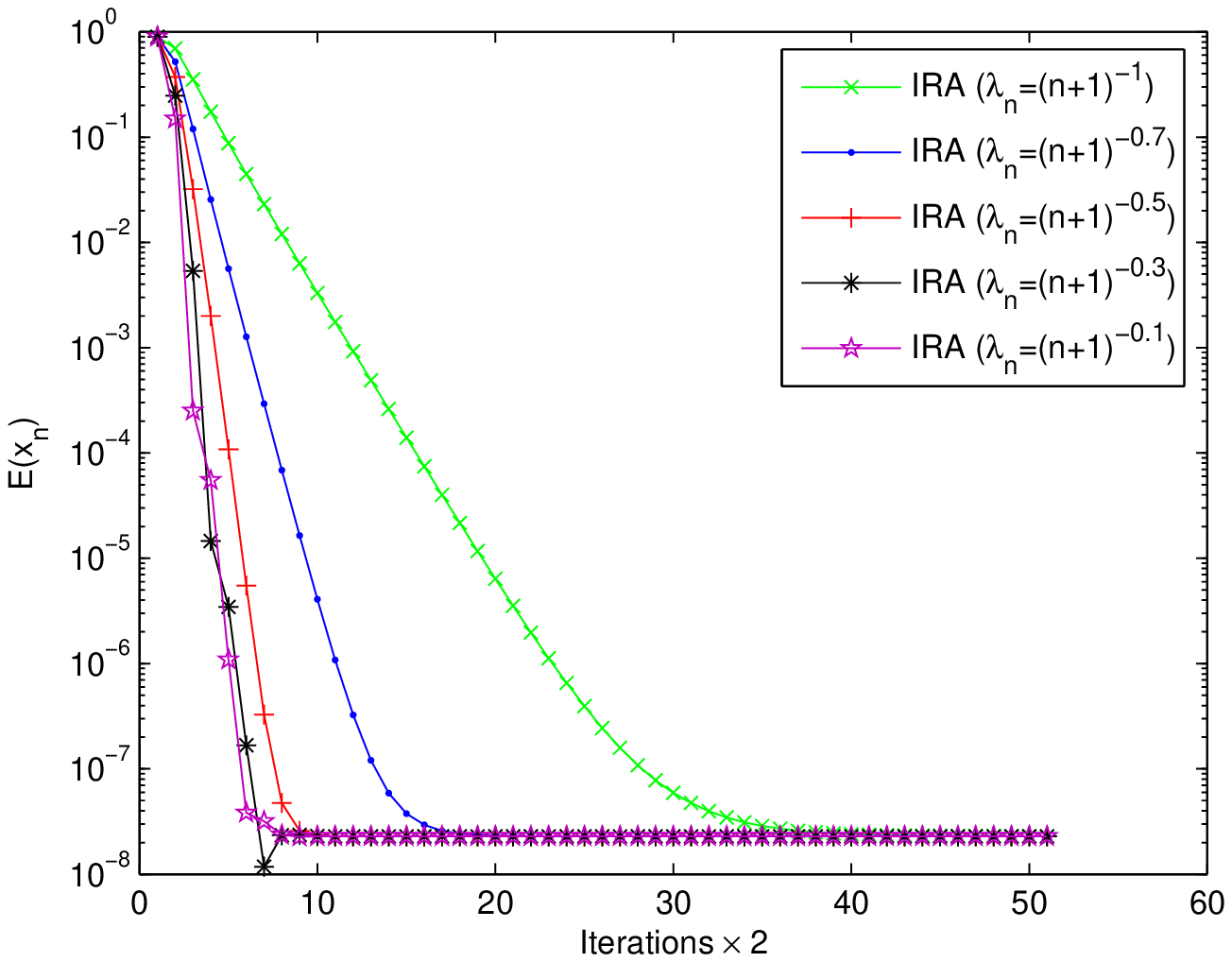}
\caption{Example 2 with $\theta_n=0.3$.}\label{fig8}
\end{minipage}
\end{figure}
%%%%%%%%%%%%%%%%%%%%%%%%%%%%%%%%%%%%%%
From the aforementioned results, we see that the convergence rate of algorithm IRA depends strictly on the convergence rate of the sequence of stepsize $\lambda_n$, 
and that algorithm IRA seems to work better when $\lambda_n$ is more slowly diminishing, and also when inertial parameter $\theta_n$ is larger. For example, in view 
of Figures \ref{fig1} and \ref{fig2}, after the first $300$ iterations, the sequence $D(x_n)$ generated by algorithm IRA with $\lambda_n=\frac{1}{n+1}$ approximates 
$10^{-7}$ while that one with $\lambda_n=(n+1)^{-0.1}=\frac{1}{\sqrt[10]{n+1}}$ is $10^{-30}$. 
%%%%%%%%%%%%%%%%%%%%%%%%%%%%%%%%%%%%%%%%%%%%%%%%%%%%%
\subsection{Compare Algorithm \ref{alg1} with others}
%%%%%%%%%%%%%%%%%%%%%%%%%%%%%%%%%%%%%%%%%%%%%%%%%%%%%
In this part, we present several experiments in comparisons algorithm IRA with others. As mentioned above, we will compare 
algorithm IRA with three algorithms having the same features as RA, EGM and M-EGM. In comparisons, we use $\theta_n=0.3$ 
for algorithm IRA, and $\lambda_n=\frac{1}{n+1}$ or $\lambda_n=\frac{1}{\sqrt[10]{n+1}}$ for all the algorithms. The starting 
points are the same as in the previous part. \\[.1in]
%%%%%%%%%%%%%%%%%%%%%%%%%%%%%%%%%%%%%%%%%%%%%
Table \ref{tab1} reports the numerial results for Example 1. In this example, since the 
solution of problem is unknown, we have used the stopping criterion as $D(x_n)\le \mbox{\rm TOL}$. As in \cite{H2017AA,H2017NUMA} 
and the previous experiments, it is seen that the convergence rate of the algorithms depends strictly on the convergence rate of sequence of 
stepsize $\lambda_n$. So, we choose here the different tolerance TOL which is based the choice of $\lambda_n$. The comparisons 
include the number of iterations (Iter.) and the execution time in second (CPU(s)). \\[.1in]
%%%%%%%%%%%%%%%%%%%%%%%%%%%%%%%%%%%%%%%%%%%%%
\begin{table}[ht]\caption{A comparison between algorithm IRA (with $\theta_n=0.3$) and others in Example 1}\label{tab1}
\medskip\begin{center}
\begin{tabular}{|c|c|c|c|c|c|c|c|c|c|c|}
 \cline{4-11}
 \multicolumn{3}{c}{} & \multicolumn{2}{|c|}{IRA (Alg. $\ref{alg1}$)} &\multicolumn{2}{c|}{RA}&\multicolumn{2}{c|}{EGM}&\multicolumn{2}{c|}{M-EGM}
\\ \hline
 $\lambda_n$& m&TOL& CPU(s)& Iter. & CPU(s)& Iter.& CPU(s)& Iter.& CPU(s)& Iter.\\ \hline
&50&$10^{-4}$&0.76&29&    1.17&47&    2.26&48&    2.34&47\\  \cline{3-11}
&&$10^{-6}$&3.47&109&    7.48&214&   14.35&221&   15.16&218\\ \cline{2-11}
$\frac{1}{n+1}$&70&$10^{-4}$&1.45&34&    2.02&54&    4.25&56&    4.41&55\\\cline{3-11}
&&$10^{-6}$&7.02&131&   14.25&256&   27.30&263&   28.32&260\\  \cline{2-11}
&100&$10^{-4}$&3.54&   37&    5.59&   64&    9.88&   67&9.83&66\\\cline{3-11}
&&$10^{-6}$&16.66&   148&    31.07&   293&    62.75&   299&66.16&297\\  
\hline
&50&$10^{-20}$&2.69&55&    4.57&95&    9.60&104&   10.09&104\\  \cline{3-11}
&&$10^{-25}$&3.41&72&    6.11&123&   11.51&134&   11.59&134\\ \cline{2-11}
$\frac{1}{\sqrt[10]{n+1}}$&70&$10^{-20}$&4.10&53&   7.34&92&   16.67&102&   17.64&102\\\cline{3-11}
&&$10^{-25}$&5.70&68&   10.81&118&   26.41&131&   25.35&131\\  \cline{2-11}
&100&$10^{-20}$&8.76&   57&    14.16&   97&    30.62&   107&30.52&106\\\cline{3-11}
&&$10^{-25}$&11.23&74&   18.95&126&   40.76&137&   40.23&137\\  
\hline
 \end{tabular}\end{center}
\end{table}

\noindent Table \ref{tab2} shows the results for Example 2. The stopping criterion is used here as $E(x_n)\le \mbox{\rm TOL}$. In view of Tables \ref{tab1} 
and \ref{tab2}, we see that algorithm IRA works the best in both number of iterations and execution time. Also, it is worth mentioning that algorithm 
IRA with inertial effects is better than the regularized algorithm RA which works without inertial term.

\begin{table}[ht]\caption{A comparison between algorithm IRA (with $\theta_n=0.3$) and others in Example 2}\label{tab2}
\medskip\begin{center}
\begin{tabular}{|c|c|c|c|c|c|c|c|c|c|c|}
 \cline{3-10}
 \multicolumn{2}{c}{} & \multicolumn{2}{|c|}{IRA (Alg. $\ref{alg1}$)} &\multicolumn{2}{c|}{RA}&\multicolumn{2}{c|}{EGM}&\multicolumn{2}{c|}{M-EGM}
\\ \hline
 $\lambda_n$& TOL& CPU(s)& Iter. & CPU(s)& Iter.& CPU(s)& Iter.& CPU(s)& Iter.\\ \hline
$\frac{1}{n+1}$&$10^{-5}$&4.21&38&    6.44&56&  13.43&63&    9.03&63\\\cline{2-10}
&$10^{-7}$&7.41&55&   10.04&83&   22.74&92&   14.10&92\\  \cline{2-10}
\hline
$\frac{1}{\sqrt[10]{n+1}}$&$10^{-5}$&0.68&8&    0.87&10&    3.72&23&    1.98&16\\\cline{2-10}
&$10^{-7}$&0.88&10&    1.35&14&    5.74&33&    2.94&24\\  \cline{2-10}
\hline
 \end{tabular}\end{center}
\end{table}
%%%%%%%%%%%%%%%%%%%%%%%%%%%%%%%%%%%%
\begin{remark}
The rate of convergence proved in Theorem \ref{theo2} shows that the smaller is the inertial parameter $\theta$, the smaller is the parameter $\alpha$ of the rate. 
Then, the convergence rate is better when the inertial parameter is not used, i.e., when $\theta=0$. This contradicts the numerical experiments presented in this section 
where the algorithm is considered with the sequence $\left\{\lambda_n\right\}$. This can be due to our \textit{bad} choice of the rate parameter $\alpha$ 
(depends on $\theta$) which originates from the analyzied techniques in the paper. This also suggests for a forthcoming work to study and reanalyze Algorithm \ref{alg1} 
where we can choose a function $\alpha=\alpha(\theta)$ which optimizes the convergence rate of the algorithm.
\end{remark}
%%%%%%%%%%%%%%%%%%%%%%%%%%%%%%%%%
\section{Conclusions} 
%%%%%%%%%%%%%%%%%%%%%%%%%%%%%%%%%
The paper has proposed a new inertial regularized algorithm for solving strongly pseudomonotone and Lipschitz-type equilibrium problems. The algorithm is a 
combination between the proximal-like regularized technique and inertial effects. By using a sequence of stepsizes being diminishing and non-summable, the proposed 
algorithm can be done without the prior knowledge of the modulus of strong pseudomonotonicity and the Lipschitz-type constant of cost bifunction. Theorem of 
strong convergence has been proved. In the case, when those constants are known, we have established the rate of linear convergence of the algorithm. Several numerical 
results have been reported to illustrate the computational performance of the algorithm in comparisons with other algorithms. These numerical results have also 
confirmed that the algorithm with inertial effects seems to work better than without inertial effects.
%%%%%%%%%%%%%%%%%%%%%%%%%%%%
\section*{Disclosure statement}
No potential conflict of interest was reported by the author.
%\section*{Funding}

\begin{acknowledgements}
The author would like to thank the Associate Editor and two anonymous referees for their valuable comments and suggestions 
which helped us very much in improving the original version of this paper. This work is supported by Vietnam National Foundation 
for Science and Technology Development (NAFOSTED) under the project: 101.01-2017.315.
\end{acknowledgements}

\end{document}